\documentclass[11pt]{amsart}
\usepackage{amsmath}
\usepackage{amsfonts}
\usepackage{amssymb}
\usepackage{amscd}
\usepackage[arrow,matrix,curve,cmtip,ps]{xy}
\usepackage{graphicx}

\numberwithin{equation}{section}

\usepackage{amsthm}

\theoremstyle{plain}
\newtheorem{thm}{Theorem}[section] 
\newtheorem{theorem}[thm]{Theorem}
\newtheorem{proposition}[thm]{Proposition}
\newtheorem{corollary}[thm]{Corollary}
\newtheorem{lemma}[thm]{Lemma}

\theoremstyle{definition}
\newtheorem{definition}[thm]{Definition}
\newtheorem{question}[thm]{Question}
\newtheorem{remark}[thm]{Remark}
\theoremstyle{remark}
\newtheorem{example}[thm]{Example}
\newtheorem*{rem}{Remark}  

\newcommand{\cl}[1]{\mathcal{#1}}
\newcommand{\bb}[1]{\mathbb{#1}}

\newcommand{\tens}{\! \otimes \!}

\newcommand{\cq}[1]{\mathbb{C}^{#1}\!/\!J}

\begin{document}

\title{The Weak Expectation Property and Riesz Interpolation}
\author{Ali S. Kavruk}
\date{}
\thanks{...}
\begin{abstract} 
We show that Lance's weak expectation property is connected to tight Riesz interpolations in lattice theory.
More precisely we first prove that if $\cl A \subset B(H)$ is a unital C*-subalgebra, where $B(H)$ is the bounded linear operators on
a Hilbert space $H$, then $\cl A$ has $(2,2)$ tight Riesz interpolation property in $B(H)$ (defined below). An extension of this requires an additional
assumption on $\cl A$: $\cl A$ has $(2,3)$   tight Riesz interpolation property in $B(H)$ at every matricial level if and only if
$\cl A$ has the weak expectation property.

Let $J= span\{(1,\!1,-1,-1,-1)\}$ in $\mathbb{C}^5$. We show that a unital C*-algebra $\cl A$ has WEP if and only if 
$\cl A \otimes_{min} (\bb C^5/J) = \cl A \otimes_{max} (\bb C^5/J)$ (here $\otimes_{min}$ and $\otimes_{max}$ are the
minimal and the maximal operator system tensor products, respectively, and $\bb C^5/J$ is the operator system quotient
of $\mathbb{C}^5$ by $J$).

We express the Kirchberg conjecture (KC) in terms of a four dimensional operator system problem. We prove that
KC has an affirmative answer if and only if $ \mathbb{C}^5/J$ has the double commutant expectation property
if and only if $ \mathbb{C}^5/J \otimes_{min} \mathbb{C}^5/J = \mathbb{C}^5/J \otimes_{c} \mathbb{C}^5/J$ (here $ \otimes_{c}$
represents the commuting operator system tensor product).

\end{abstract}
\maketitle

We continue our research on finite dimensional operator systems by means of recently developed
quotient, tensor and nuclearity theory \cite{kptt}, \cite{kptt2}, \cite{Kav2}. The main purpose of the present paper can be divided into three parts.
Letting
$$
J = span\{(1,1,-1,-1,-1)\} \subset \bb C^{5},
$$
the operator system quotient $\bb C^5 / J$ (which is different than the operator space quotient) can be identified
(unitally and completely order isomorphically) with an operator subsystem
of the full group C*-algebra $C^*(\bb Z_2 * \bb Z_3)$. By using this identification we first obtain a new weak expectation
property (WEP) criteria:
\begin{theorem}
 A unital C*-algebra $\cl A$ has WEP if and only if we have the unital and complete order isomorphism
$$
\cl A \otimes_{min} (\bb C^5 / J) = \cl A \otimes_{max} (\bb C^5 / J).
$$
\end{theorem}
This WEP criteria also allows us to re-express Kirchberg's conjecture \cite{Ki2} in terms of a problem about this four dimensional operator system.
Recall that Kirchberg's conjecture asserts that every separable C*-algebra that has the local lifting property (LLP) has WEP. In \cite{kptt2}
this problem is approached in the operator system setting and the following reformulation is given:
Every finite dimensional operator system that has the lifting property (LP) has the double commutant expectation
property (DCEP). (DCEP is one extension of WEP from unital C*-algebras to general operator systems and discussed briefly in Subsection \ref{sec DCEP} below).
In \cite{Kav2} it was also proven that this is equivalent to the universal operator system generated by two contractions, $\cl S_2$, (which has LP) has DCEP.
One of our main results in Section 5 is the following four dimensional reduction.

\begin{theorem}
The following are equivalent:
\begin{enumerate}
 \item Kirchberg's conjecture has an affirmative answer.
 \item $\bb C^5 / J$ has DCEP.
 \item We have the complete order isomorphism 
$$
(\bb C^5 / J) \otimes_{min} (\bb C^5 / J) = (\bb C^5 / J) \otimes_{c} (\bb C^5 / J). 
$$
\end{enumerate}
\end{theorem}
\noindent Here $\otimes_c$ denotes the (maximal) commuting tensor product and we briefly summarized its main properties in Section 2. 

In the final section we examine the role of WEP in tight Riesz separation properties and we prove, for larger arguments,
that these two concepts are identical.
Let $\cl A$ be a unital C*-subalgebra of $B(H)$. We say
that $\cl A$ has the \textbf{(k,m) tight Riesz interpolation property in $B(H)$}, TR(k,m)-property in short, if for every 
self-adjoint elements $x_1,...,x_k$ and $y_1,...,y_m$ in $\cl A$
whenever there is an element $b$ in $B(H)$ with
$$
x_1,...,x_k < b < y_1,...,y_m
$$
then there is an element $a$ in $\cl A$ such that
$$
x_1,...,x_k < a < y_1,...,y_m.
$$
Here $x<y$ stands for $\delta I \leq y-x$ for some positive $\delta$ where $I$ denotes the unit of $B(H)$. Likewise,
we say that $\cl A$ has the \textbf{complete TR(k,m)-property in} $B(H)$ if $M_n(\cl A)$ has TR(k,m)-property in $M_n(B(H))$
for every $n$. We first prove the following:
\begin{theorem}
$\cl A \subset B(H)$ has the complete TR(2,2)-property in $B(H)$. 
\end{theorem}
While (2,2)-interpolation is automatically satisfied, higher interpolations require an additional hypothesis on $\cl A$. The following
is our main result in Section 7:
\begin{theorem}
Let $\cl A \subset B(H)$ is a unital C*-algebra. Then the following are equivalent:
\begin{enumerate}
 \item $\cl A$ has the weak expectation property;
 \item $\cl A$ has the complete TR(2,3)-property in $B(H)$;
 \item $\cl A$ has the complete TR(k,m)-property in $B(H)$ for some $k \geq 2$, $m \geq 3$;
 \item $\cl A$ has the complete TR(k,m)-property in $B(H)$ for all positive integers $k$ and $m$.
\end{enumerate}
\end{theorem}

This characterization of WEP is independent of the particular faithful representation
of the C*-algebra on a Hilbert space. (In fact a careful reading of our proof indicates that we can even consider
an (operator systematic) representation.)  In this respect it differs
from Lance's original definition of WEP \cite{La}, which requires that every faithful representation
of the given C*-algebra has a conditional expectation into its double commutant.
This requirement in the original definition is essential since every C*-algebra
has at least one representation which has a a conditional expectation into its double
commutant.


$ $

Our proofs makes use of recently developed tensor, quotient and nuclearity theory of operator systems. In this regard
we start with a brief overview on the operator systems. The preliminary section includes basic facts on duality, quotients,
C*-covers etc. We devote Section 2 for the tensor products in the category of operator systems. Here, after the axiomatic
definition of tensor products we briefly summarized main properties of the minimal (min), maximal (max), commuting (c) and two asymmetric
tensor products, enveloping left (el) and enveloping right (er). The set of all tensor products admits a natural (partial) order and
the primary tensor products we consider have the following pattern:
$$
\otimes_{min} \;\; \leq \;\;  \otimes_{el} \;\;, \;\; \otimes_{er} \;\; \leq \;\; \otimes_c \;\; \leq \;\; \otimes_{max}.
$$ 
Section 3 includes several nuclearity related properties including the (operator system) local lifting property (osLLP), double
commutant expectation property (DCEP), weak expectation property and exactness. These operator system notions, along with
tensor characterizations, studied in \cite{kptt2}.
The term ``nuclearity related'' perhaps best seen in the operator system setting: Given tensor products $\alpha \leq \beta$
we call an operator system $(\alpha, \beta)$-nuclear if
$$
\cl S \otimes_{\alpha} \cl T = \cl S \otimes_{\alpha} \cl T \mbox{ for every operator system } \cl T.
$$
Our main purpose in Section 3 is to exhibit the following ``nuclearity diagram'':
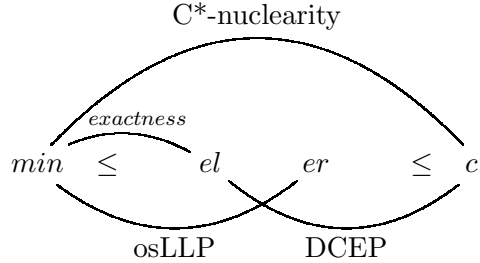
$$\hspace{4cm}
\xymatrix{
\;\;\;\;\;\;\;\; min \;\;\; \leq \ar@{-}@/^1pc/[r]^{\;\;\;\;\;exactness}  
\ar@{-}@/^4pc/[rrr]^{\mbox{C*-nuclearity}} 
\ar@{-}@/_2pc/[rr]_{\mbox{osLLP}}  
  &
el    \ar@{-}@/_2pc/[rr]_{\mbox{DCEP}} 
   & er & \leq \;\;\;  c
\hspace{1cm}&  & &
}
\hspace{4cm}
\xymatrix{
\;\;\;\;\;\;\;\; min \;\;\; \leq \ar@{-}@/^1pc/[r]^{\;\;\;\;\;exactness}  
\ar@{-}@/^4pc/[rrr]^{\mbox{C*-nuclearity}} 
\ar@{-}@/_2pc/[rr]_{\mbox{osLLP}}  
  &
el    \ar@{-}@/_2pc/[rr]_{\mbox{DCEP}} 
   & er & \leq \;\;\;  c
\hspace{1cm}&  & &
}
$$
For example, an operator system $\cl S$ is exact if and only if it is (min,el)-nuclear. The remaining of the Section 3 includes several
interesting examples on the stability of these properties under (operator system) quotients, duality etc. For example,
in contrast to C*-algebra ideal quotients, exactness is not preserved under operator system quotients. Though, in the finite dimensional case, the
lifting property is preserved under quotients by null subspaces, etc.

In Section 4 we recall basic facts on the coproducts of operator systems introduced by Kerr and Li \cite{KL}. This should
be considered as operator system variant of the unital free products of C*-algebras. The main purpose of this
section is to obtain the following identification: Consider $\bb Z_2 * \bb Z_3 = \langle a,b: a^2 = b^3 = e \rangle$. Let
$\lambda$ be the universal representation  of  $\bb Z_2 * \bb Z_3$ in its full group C*-algebra $C^*(\bb Z_2 * \bb Z_3)$.
Let 
$$
\cl S = span\{\lambda(e), \lambda(a), \lambda(b), \lambda(b)^*\} \subset C^*(\bb Z_2 * \bb Z_3).
$$
Then the last theorem of the section indicates that we have the unital complete order isomorphism
$$
\bb C^5 / J \cong \cl S.
$$

Section 5 contains two of our main application that we pointed out at the beginning of introduction, namely the four dimensional
version of WEP criteria and four dimensional operator system variant of Kirchberg's conjecture.

$ $

Fixing the basis $\{\dot{e}_1,\dot{e}_2,\dot{e}_3,\dot{e}_4\}$ of $\bb C^5/J$ every element in the algebraic
tensor $\cl S \otimes (\bb C^5 / J)$ can be uniquely written as
\begin{eqnarray}\label{fgh}
s_1 \tens \dot{e}_1 + s_2 \tens \dot{e}_2 + s_3 \tens \dot{e}_3 + s_4 \tens \dot{e}_4.
\end{eqnarray}
Section 6 is devoted to understand the positivity criteria in $- \otimes_{min}  (\bb C^5 / J)$, $- \otimes_{c}  (\bb C^5 / J)$, and $- \otimes_{max}  (\bb C^5 / J)$.
So simply put, when the above mentioned element is positive when the algebraic tensor  $\cl S \otimes (\bb C^5 / J)$
is equipped with the minimal, commuting or the maximal tensor product. As a rehearsal we would like give
the following part of Proposition \ref{prop pos-criteria}.  Let $\cl S \subset B(H)$ be an operator subsystem and let $u$ denote the expression in (\ref{fgh}).
Then
\begin{enumerate}
 \item $u > 0$ in $\cl S \otimes_{min} (\bb C^5 / J)$ if and only if there is an element $b \in B(H)^+$ such that 
$$
s_1,s_2 > b \mbox{ and } s_3, s_4 \geq -b.
$$
 \item $u > 0$ in $\cl S \otimes_{max} (\bb C^5 / J)$ if and only if there is an element $s \in \cl S^+$ such that 
$$
s_1,s_2 > s \mbox{ and } s_3, s_4 \geq -s.
$$
\end{enumerate}

Several positivity criteria that we obtain in Section 6 together with the WEP characterization given in Section 5 form the basic
part of Section 7 and allows us to prove the equivalence of WEP and (k,m) tight Riesz interpolations when $k \geq 2$ and $m \geq 3$.

\section{Preliminaries}

In this section we establish the terminology and state the definitions and
basic results that shall be used throughout the paper. It is assumed throughout
that all C*-algebras are unital and all ideals are closed and two sided (and consequently $*$-closed).
The C*-algebra of $n\times n$ matrices is denoted by $M_n$. 
By an operator system $\cl S$ we mean a unital, $*$-closed subspace of $B(H)$
together with the induced matricial order structure.
We refer the reader to \cite{Pa} for an introductory exposition of these objects along with
their abstract characterization due to Choi and Effros. If $\varphi : \cl S \rightarrow \cl T$
is a linear map, where $\cl T$ and $\cl S$ are operator systems, the $n^{th}$-amplification
$\varphi^{(n)}$ is defined by $\varphi \otimes id_n : \cl S \otimes M_n \rightarrow \cl T \otimes M_n$.
$\varphi$ is called $n$-positive if $\varphi^{(n)}$ is positive and completely positive (cp) if $\varphi^{(n)}$
is positive for all $n$. If in addition $\varphi(e) = e$, i.e. if it maps the unit to unit, then we say that
$\varphi$ is unital and completely positive (ucp).

$ $

A pair $(i,\cl A)$ is called a C*-cover of an operator system $\cl S$ if $i: \cl S \rightarrow \cl A$
is a unital complete order embedding such that $i(\cl S)$ (topologically) generates $\cl A$ as a C*-algebra.
We often identify $\cl S$ with $i(\cl S)$ and consider it as an operator subsystem of $\cl A$.
Every operator system $\cl S$ 
admits two special C*-covers, the universal and the enveloping
C*-algebras, denoted by $C_u^*(\cl S)$ and $C^*_e(\cl S)$, respectively. The universal C*-algebra has the
following ``maximality'' property: For every ucp map $\varphi: \cl S \rightarrow \cl B$, where $\cl B$ is a C*-algebra,
there is a uniquely determined unital $*$-homomorphism $\pi: C^*_u(\cl S) \rightarrow \cl B$ which extends $\varphi$.
The enveloping C*-algebra is the ``minimal'' C*-cover in the sense that for any C*-cover $(i,\cl A)$
of $\cl S$ there is a uniquely determined unital $*$-homomorphism $\pi : \cl A \rightarrow C^*_e(\cl S)$
such that $\pi(i(s)) = s$ for every $s$ in $\cl S$. The enveloping C*-algebra of $\cl S$ can be identified with the C*-algebra
generated by $\cl S$ in its injective envelope $I(\cl S)$. The reader may refer to \cite[Chp. 15 ]{Pa} for an excellent
survey on the injectivity of operator systems. (However, for convenience, we remark that
every injective operator system has the structure of a C*-algebra.)

\subsection{Duality} The duality in the operator system, especially in the finite dimensional case, 
has had a substantial role in the study of tensor products. Starting with an operator system $\cl S$
the Banach dual $\cl S$ has a natural matricial order structure.
For $f$ in $\cl S^d$, the involution is defined by $f^*(s) = \overline{f(s^*)}$. The matricial order structure is
described as:
$$
(f_{ij}) \in M_n(\cl S^d) \mbox{ is positive if the map } \cl S \ni
s \mapsto (f_{ij}(s)) \in M_n \mbox{ is cp}. 
$$
Throughout the paper $\cl S^d$ will always represent this
matrix ordered vector space. The bidual Banach space $\cl S^{dd}$ has also
a natural matricial order structure arising from the fact that it is the dual of $\cl S^d$. The following
is perhaps well known, see \cite{kptt}, e.g.:
\begin{theorem}
$\cl S^{dd}$ is an operator system with unit $\hat{e}$, the canonical image of $e$ in $\cl S^{dd}$. 
Moreover, the canonical embedding of $\cl S$ into $\cl S^{dd}$ is a complete order embedding.
\end{theorem}

A state $f$ on $\cl S$ (i.e. a positive linear functional with $f(e)=1$) is called \textbf{faithful} if $s \geq 0$ and $f(s) = 0$
implies that $s = 0$, in other words $f$ maps non-zero positive elements to non-zero positive scalars. Every
finite dimensional operator system $\cl S$ possesses a faithful state \cite[Sec. 6]{CE2}, and consequently, the matricially
ordered space $\cl S^d$ is again an operator system with the (non-canonical) Archimedean matrix order unit $f$.

\subsection{Quotients} A subspace $J$ of an operator system $\cl S$ is called a \textbf{kernel}
if it is the kernel of a ucp map defined from $\cl S$ into another operator system $\cl T$.
Note that $J$ has to  be a $*$-closed and non-unital subspace of $\cl S$, however, these
properties, in general, do not characterize a kernel. When $J \subset \cl S$ is a kernel
the algebraic quotient $\cl S / J$ has a natural operator system structure with unit $\dot{e} = e + J$.
We first define
$$
D_n = \{(\dot{s}_{ij}) : \;\; (s_{ij})\mbox{ is positive in } M_n(\cl S)\}.
$$
$\cl S / J$ together with cones $\{D_n\}_{n=1}^\infty$ form a matricially ordered space, moreover,
$\dot e$ can be shown to be a matrix order unit. However $D_n$ may fail to be closed in the order topology
induced by $\dot e_n$ and therefore another step, namely the completion of the cones, (also known as the Archimedeanization process)
is required:
$$
C_n = \{(\dot{s}_{ij}) : \;\; (\dot s_{ij}) + \epsilon \dot{e}_n \mbox{ is  in }  D_n \mbox{ for all } \epsilon > 0\}.
$$
Now $\cl S / J$ with matricial order structure  $\{C_n\}_{n=1}^\infty$ and unit $\dot e$ is an operator system
and is called the quotient operator system of $\cl S$ by $J$. $J$ is said to be \textbf{proximinal} if $D_1 = C_1$ and
\textbf{completely proximinal} if $D_n = C_n$.

\begin{remark}
A finite dimensional subspace $J$ of an operator system $\cl S$ is a called a \textbf{null-subspace} if it is closed
under the involution $*$ and does not contain any positive other than 0. In \cite{Kav} it was shown that
every null subspace is a completely proximinal kernel.
\end{remark}

\begin{example}
Let $J_n \subset M_n$ be the set of diagonal operators with 0-trace. Clearly $J_n$ is a null-subspace. So it is a
completely proximinal kernel.
\end{example}

\begin{example}
Let $y \in \cl S$ be a self-adjoint element that is neither positive nor negative. Then $J = span\{y\}$
is one dimensional null-subspace of $\cl S$.
\end{example}

\begin{example}
Let $\bb F_n$ be the free group on $n$-generators. Let $C^*(\bb F_n)$ be the full group C*-algebra of $\bb F_n$.
Consider $J = span\{u_1,...,u_n,u_1^*,...,u_n^*\}$ where $u_1,...,u_n$ are the universal unitaries. 
Then $J$ is a null subspace and therefore a completely proximinal
kernel in  $C^*(\bb F_n)$.
\end{example}

The operator system quotients have the following compatibility property with the morphisms: Letting
$J \subset \cl S$ be a kernel, if $\varphi: \cl S \rightarrow \cl T$ is cp map with $J \subset ker(\varphi)$ then
the induced map $\bar \varphi: \cl S / J \rightarrow \cl T$ is still a cp map. Conversely if $\psi :  \cl S / J \rightarrow \cl T$
is a cp map then $\psi \circ q : \cl S \rightarrow \cl T$, where $q:  \cl S \rightarrow \cl S / J$ is the quotient map,
is a cp map with kernel including $J$.

A surjective cp map $\varphi: \cl S \rightarrow \cl T$ is called a \textbf{complete quotient map} if the induced map
$\bar \varphi : \cl S /ker(\varphi) \rightarrow \cl T$, which is cp, is a complete order isomorphism. These maps are
the dual notions of the complete order embeddings. Following is from \cite[Sec. 2]{Kav2}.

\begin{theorem}
Let $\cl S$ and $\cl T$ be finite dimensional operator systems. If $i: \cl S \rightarrow \cl T $ is a complete order
embedding then the adjoint map $i^d : \cl T^d \rightarrow \cl S^d$ is a complete quotient map. Moreover
by special selection of faithful states on $\cl S$ and $\cl T$ one may suppose that $i^d$ is also unital and the kernel
of $i^d$ is a null-subspace of $\cl T^d$.
\end{theorem}

\noindent So roughly speaking if $\cl S \subset \cl T$ then we get $\cl S^d = \cl T^d / J$ for some null-subspace $J \subset \cl T^d$. A moment
of thought shows that $J$ has to be the collection of linear functional that vanish on $\cl S$. The converse of this also true \cite{pf}.

\begin{theorem}[Farenick, Paulsen]
Let $q: \cl S \rightarrow \cl T$ be a complete quotient map. Then $q^d : \cl T^d \rightarrow \cl S^d$ is a complete order
embedding.
\end{theorem}

\noindent A \textit{unitality} problem of $q^d$ may occur in this case. We need the kernel of $q$ to be a null-subspace
to be able to assume $q^d$ is unital (by proper selections of faithful states on $S$ and $\cl T$). This remark together with the 
above exhibit how null-subspaces occur naturally.

\section{Tensor Products of Operator Systems}

In this section we recall the axiomatic definition of tensor products in the
category of operator systems and review properties of several tensor
products established in \cite{kptt}. Suppose $\cl S$ and $\cl T$ are two
operator systems. A matricial cone structure $\tau = \{C_n\}$ on $\cl S \otimes
\cl T$  where $C_n \subset M_n(\cl S \otimes \cl T)_{sa}$,  is called a
tensor product structure if \begin{enumerate}

\item $(\cl S \otimes \cl T,\{C_n\},e_{\cl S}\otimes e_{\cl T})$ is an operator
system,

\item for any $(s_{ij}) \in M_n(\cl S)^+$ and $(t_{rs}) \in M_k(\cl T)^+$,
$(s_{ij} \otimes t_{rs})$ is in $C_{nk}$ for all $n,k$,

\item if $\phi: \cl S \rightarrow M_n$ and $\psi: \cl T \rightarrow M_k$ are ucp
maps then $\phi\otimes\psi:\cl S \otimes \cl T \rightarrow M_{nk}$ is a ucp map
for every $n$ and $k$.
\end{enumerate}
A mapping $\tau:\cl O \times \cl O\rightarrow \cl O$ is said to be an
\textit{operator system tensor product} (or simply a \textit{tensor product})
provided $\tau$ maps each pair $(\cl S,\cl T)$ to a a tensor product structure
on $\cl S \otimes \cl T$, denoted by $\cl S \otimes_{\tau} \cl T$. A tensor
product $\tau$ is said to be \textbf{functorial} if for every operator systems
$\cl S_1, \cl S_2, \cl T_1$ and $\cl T_2$ and every ucp maps $\phi: \cl
S_1\rightarrow \cl S_2$ and $\psi: \cl T_1\rightarrow \cl T_2$ the associated
map $\phi \otimes \psi: \cl S_1 \otimes_{\tau} \cl T_1 \rightarrow \cl S_2
\otimes_{\tau} \cl T_2$ is ucp. A tensor product $\tau$ is called
\textit{symmetric} if $\cl S \otimes_{\tau} \cl T = \cl T \otimes_{\tau} \cl S $
and \textit{associative} if $(\cl S \otimes_{\tau} \cl T)\otimes_{\tau} \cl R =
\cl S \otimes_{\tau} (\cl T \otimes_{\tau} \cl R) $ for every $\cl S, \cl T$ and
$\cl R$.

$ $

There is a natural partial order on the operator system tensor products: If
$\tau_1$ and $\tau_2$ are two tensor products then we say that $\tau_1 \leq
\tau_2$ if for every operator systems $\cl S$ and $\cl T$ the identity $id: \cl
S \otimes_{\tau_2} \cl T \rightarrow \cl S \otimes_{\tau_1} \cl T$ is completely
positive. In other words $\tau_1$ is smaller with respect to $\tau_2$ if the
cones it generates are larger. (Recall that larger matricial cones generate
smaller canonical operator space structure.) The partial order on operator
system tensor products forms a lattice as pointed out in \cite[Sec. 7]{kptt}
and raises fundamental nuclearity properties as we shall discuss in the next
section.

$ $

In the remaining of this section we discuss several important tensor products,
namely the minimal (min), maximal (max), maximal commuting (c), enveloping left
(el) and enveloping right (er) tensor products. With respect to the partial
order relation given in the previous paragraph we have the following schema
\cite{kptt} :
$$
min \;\; \leq \;\; el\;\;  ,\;\;  er\;\; \leq \;\; c \;\; \leq \;\; max.
$$

\subsection{Minimal Tensor Product} Let $\cl S$ and $\cl T$ be two operator 
systems. We define the matricial cone
structure on the tensor product $\cl S \otimes \cl T$ as follows:
\begin{align*}
C_n^{min}(\cl S,\cl T) = \{(u_{ij}) &\in M_n(\cl
S\otimes \cl T) : ((\phi\otimes\psi)(u_{ij}))_{ij} \in
M_{nkm}^+ \\
& \mbox{ for every ucp maps } \phi : \cl S \to M_k  \mbox{ and }  \psi :  \cl
T \to M_m \mbox{ for all $k,m$.}\}.
\end{align*} 
The matricial cone structure $\{C_n^{min}\}$ satisfies the axioms (1), (2), and
(3) and the resulting operator system is denoted by $\cl S \otimes_{min} \cl
T$. For the proofs of the following we refer the reader to  \cite[Sec. 4]{kptt}.

\begin{enumerate}
 \item If $\tau$ is another operator system structure on $\cl S \otimes \cl T$ then
we have that $min \leq \tau$. In other words $\{C_n^{min}\}$ forms the
largest cone structure.

 \item The minimal tensor product, when considered
as a map $min : \cl O \times \cl O \rightarrow \cl O$, is symmetric,
associative and functorial.

 \item \textit{min} is injective in the sense that if $\cl S_1
\subset \cl S_2$ and $\cl T_1 \subset \cl T_2$ then $\cl S_1 \otimes_{min} \cl
T_1 \subset \cl S_2 \otimes_{min} \cl T_2$ completely order isomorphically.

\item \textit{min} is \textit{spatial}, that is, if
$\cl S \subset B(H)$ and $\cl T \subset B(K)$ then the concrete operator system
structure on $\cl S \otimes \cl T$ arising from the inclusion $B(H\otimes K)$
coincides with the minimal tensor product. From this one easily derives that $min$
coincides with the the C*-algebraic minimal tensor products when restricted to
unital C*-algebras (except for completion). 

\end{enumerate}

\subsection{Maximal Tensor Product}\label{subsec max} 
The construction of the maximal tensor product of two operator systems $\cl S$
and $\cl T$ requires two steps. We first define
$$
D_n^{max}(\cl S,\cl T) = \{A^* (P\otimes Q) A :
P\in M_k(\cl S)^+, Q\in M_m(\cl T)^+, A \in M_{km,n}, \\ k,m\in \bb{N}\}.
$$
Although the matricial order structure $\{ D_n^{max}\}$ is strict and compatible (for the
definitions see \cite[Chp. 13]{Pa} e.g.) it might not be closed with respect to the
order topology and hence another step, namely the completion of the cones, is required.
Since after this step $e_n$ is an Archimedean order unit this process is also known as the
Archimedeanization process (see \cite{pt} e.g).  We define
$$
C_n^{max}(\cl S,\cl T) = \{ P \in M_n (\cl S \otimes \cl T) : r(e_1 \otimes
e_2)_n +P \in D_n^{max}(\cl S, \cl T) \text{ $\forall$ $r>0$} \}.
$$

Now the matrix order structure $\{ C_n^{max}\}$ satisfies all the axioms and the
resulting operator system is denoted by $\cl S \otimes_{max} \cl T$. Below we listed the main properties
of this tensor product:

\begin{enumerate}

 \item Let $\tau$ is another operator system structure on $\cl S \otimes \cl T$ then
$\tau \leq max $, that is,  $\{C_n^{max}\}$ is the smallest cone
structure. 

 \item \textit{max}, as min, has all properties symmetry, associativity and
functoriality. Moreover, like min, it has the strong functoriality in the sense
that if $\varphi_1 : \cl S _i \rightarrow \cl T_i$ are cp maps then the associative
tensor map $\varphi_{1}\otimes \varphi_2  :\cl S_1\otimes_{max} \cl S_2 \rightarrow \cl T_1 \otimes_{max} \cl T_2$
is again cp.

 \item \textit{max} coincides with the C*-algebraic maximal tensor product when
restricted to unital C*-algebras (again, except for completion). 

 \item As it is well known from C*-algebras, max does not have  the injectivity property that
min possesses. However it is projective as shown by Han \cite{Han}: if $q_1 : \cl S_1 \rightarrow \cl T_1$
and $q_2 : \cl S_2 \rightarrow \cl T_2$ are complete quotients maps then the tensor map
$$
q_1 \otimes q_2  :\cl S_1\otimes_{max} \cl S_2 \rightarrow \cl T_1 \otimes_{max} \cl T_2
$$
is again a complete quotient map.

 \item Lance's duality result regarding the maximal tensor products for C*-algebras in  \cite{La2} can be extended to
general operator systems: A linear map $f : \cl S \otimes_{max} \cl T \rightarrow \mathbb{C}$
is positive if and only if the corresponding map $\varphi_f : \cl S \rightarrow \cl T^d$ is completely
positive. Here $\varphi_f(s)$ is the linear functional on $T$ given by $\varphi_f(s)(t) = f(s \otimes t)$.
(See also \cite[Lem. 5.7 and Thm. 5.8]{kptt}.) Consequently we obtain the following representation 
of the maximal tensor product:
$$
(\cl S \otimes_{max} \cl T)^{d,+} = CP(\cl S, T^d).
$$
\end{enumerate}

\subsection{(Maximal) Commuting Tensor Product}\label{subsec c}
Another important tensor product we want to discuss is the commuting (or maximal commuting)
tensor product which is denoted by c. It agrees with the C*-algebraic maximal tensor
products on the category of unital C*-algebras however it is different then max
for general operator systems. We define the matricial order structure by using the
ucp maps  with commuting ranges. More precisely, if $\cl S$ and $\cl T$ are two
operator systems then $C^{com}_n$ consist of all $(u_{ij}) \in M_n(\cl S\otimes
\cl T)$ with the property that for any Hilbert space $H$,  any ucp $\phi:\cl S
\rightarrow B(H)$ and $\psi:\cl T \rightarrow B(H)$ with commuting ranges
$$
(\phi\cdot \psi)^{(n)} (u_{ij}) \geq 0
$$
where $\phi \cdot \psi: \cl S \otimes \cl T \rightarrow B(H)$ is the map defined
by $\phi \cdot \psi(s\otimes t) = \phi(s)\psi(t)$. The matricial cone
structure $\{C^{com}_n\}$ satisfies the axioms (1), (2) and (3), and the
resulting operator system is denoted by $\cl S \otimes_{c} \cl T $. We again list the main
properties of this tensor product:

\begin{enumerate}
 \item The commuting tensor product $c$ is functorial and symmetric however we don't know
whether is it associative  or not. Though for every $n$ we have
$$
M_n(\cl S \otimes_{c} \cl T) = M_n(\cl S) \otimes_{c} \cl T = \cl S \otimes_{c}M_n( \cl T).
$$
\item  If $\tau$ is an operator
system structure on $\cl S \otimes \cl T$ such that $\cl S \otimes_{\tau} \cl T$
attains a representation in a $B(H)$ with ``$\cl S$'' and ``$\cl T$'' portions
are commuting then $\tau \leq c$. This directly follows from the definition of c
and justifies the name ``maximal commuting''.

\item As we pointed out $c$ and $max$ coincides on the unital C*-algebras. This result can be extended
even further: If $\cl A$ is a unital $C^*$-algebra and $\cl S$ is an operator
system, then $ \cl A \otimes_{c} \cl S = \cl A \otimes_{\max} \cl S. $

\item For every $\cl S$ and $\cl T$ we have the unital complete order embedding
$$
\cl S \otimes_{c} \cl T \subset C^*_u(\cl S)  \otimes_{max} C^*_u(\cl T).
$$
\item The ucp maps defined by the commuting tensor product are the compression of
the ucp maps with commuting ranges, that is, if  $\varphi: \cl S \otimes_{c} \cl T \rightarrow B(H)$ is a ucp map
Then there is Hilbert space $K$ containing $H$ as a Hilbert subspace and ucp maps $\phi: \cl S \rightarrow B(K)$ and
 $\psi: \cl T \rightarrow B(K)$ with commuting ranges such that $\varphi = P_H \phi \cdot \psi |_H $. Conversely, every
such map is ucp.
\end{enumerate}

\subsection{Some Asymmetric Tensor Products}

In this subsection we discuss the enveloping left (el) and enveloping right 
(er) tensor products. Given operator systems $\cl S$ and $\cl T$ we define 
$$
\cl S \otimes_{el} \cl T : \subseteq I(\cl S) \otimes_{max} \cl T \mbox{ and }
\cl S \otimes_{er} \cl T : \subseteq \cl S \otimes_{max} I(\cl T)
$$  
where $I(\cdot)$ is the injective envelope of an operator system. We remark
that $ I(\cl S) \otimes_{c} \cl T =  I(\cl S) \otimes_{max} \cl T$ as the
$I(\cl S)$ has the structure of a C*-algebra. Here are main properties:

\begin{enumerate}
 \item \textit{el} and \textit{er} are functorial tensor products. We don't
know whether these tensor products are associative or not. However for every $n$ we have
$$
M_n(\cl S \otimes_{el} \cl T) = M_n(\cl S) \otimes_{el} \cl T = \cl S \otimes_{el} M_n( \cl T).
$$
A similar ``associativity with matrix algebras'' holds for \textit{er} too.

\item Both \textit{el} and \textit{er} are not symmetric but they are asymmetric in the sense that 
$$ \cl S \otimes_{el} \cl T
=\cl T \otimes_{er} \cl S \;\;\; \mbox{ via the map } \;\;\; s \otimes t \mapsto
t \otimes s.
$$ 

\item The tensor product \textit{el} is the maximal left injective functorial tensor
product, that is, for any $\cl S \subset \cl S_1$ and $\cl T$ we have
$$
\cl S \otimes_{el} \cl T \subseteq \cl S_1 \otimes_{el} \cl T
$$
and it is the maximal functorial tensor product with this property. Likewise, er is the maximal right injective
tensor product.

\item Tensor product $el$ is independent of the injective operator system that containing $\cl S$ as an operator subsystem. For example
if $\cl S \subset B(H)$ then we have the complete order embedding $\cl S \otimes_{el} \cl T \subseteq B(H) \otimes_{max} \cl T$. This simply
follows from the left injectivity of $el$ and the fact that $el$ and and $max$ coincides if the left tensorant is an injective operator system
which directly follows from the definition.

\item $el$ and $er$ are in general not comparable but they both lie between $min$ and $c$.
\end{enumerate}

$ $

\section{Characterization of Various Nuclearities}\label{sec nuclearity}

In the previous section we have reviewed the tensor products in the category of
operator systems. In this section we will overview the behavior of the
operator systems under tensor products. More precisely, we will see several
characterizations of the operator systems that fix a pair of tensor products.

$ $

Given two tensor products $\tau_1 \leq \tau_2 $, an operator systems $\cl S$ is
said to be\textit{ $(\tau_1,\tau_2)$-nuclear} provided $ \cl S \otimes_{\tau_1}
\cl T = \cl S \otimes_{\tau_2} \cl T $ for every operator system $\cl T$. 
We remark that the \textit{place} of the operator system $\cl S$ is important as
not all the tensor products are symmetric.

\subsection{Operator System Local Lifting Property (osLLP)}\label{sec osLLP} 
We want to start with a discussion of (operator system) local lifting property
(osLLP) which characterizes the operator systems having (min,er)-nuclearity. 

\begin{definition}
An operator system $\cl S$ is said to have \textbf{osLLP} if for every unital C*-algebra $\cl
A$ and ideal $I$ in $\cl A$ and for every ucp map $\varphi: \cl S \rightarrow
\cl A / \cl I$ the following holds: For every finite dimensional operator
subsystem $\cl S_0 $ of $\cl S$, the restriction of $\varphi$ on $\cl S_0$, say
$\varphi_0$, lifts to a completely positive map on $\cl A$ so that the following diagram
commutes (where $q: \cl A \rightarrow \cl A / I$ is the quotient map).
$$
\xymatrix{
   &     &  \cl A \ar[d]^q \\
\hspace{2cm}\cl S_0 \subset \cl S \ar[rr]_-{ucp \; \varphi} 
\ar@{.>}[rru]^{\tilde{\varphi}_0} &     &   \cl A/ \cl I} \hspace{4cm}
$$
\end{definition}

Of course, $\cl S$ may possess osLLP without a global lifting. We also remark
that the completely positive local liftings can also be chosen to be ucp  in the
definition of osLLP (see the discussion in \cite[Sec. 8]{kptt2}). The LLP
definition for a C*-algebra given in \cite{Ki2} is the same. So it follows that
a unital C*-algebra has LLP (in the sense of Kirchberg) if and only if it has
osLLP. We can now state the connection of osLLP and tensor products given in \cite{kptt2}:

\begin{theorem}\label{thm LLPchar} 
The following are equivalent for an operator system $\cl S$:
\begin{enumerate}
 \item $\cl S$ has osLLP.
 \item $\cl S \otimes_{min} B(H) = \cl S \otimes_{max} B(H)$ for every Hilbert
space $H$ $($or for $H = l^2(\mathbb{N}))$.
 \item $\cl S$ is (min,er)-nuclear, that is, $\cl S \otimes_{min} \cl T = \cl S
\otimes_{er} \cl T$ for every $\cl T$.
\end{enumerate}
\end{theorem}

\noindent Note that if $\cl A$ is a C*-algebra then the equivalence of (1) and (2) recovers
a well known result of Kirchberg \cite{Ki2}. If we let $\mathbb{B}$ denote $B(l^2(\mathbb{N}))$, 
the above equivalent conditions, in some similar context,
is also called $\mathbb{B}$-nuclearity. (See \cite{blecher}, e.g.)
Consequently for operator systems osLLP, $\mathbb{B}$-nuclearity and
(min,er)-nuclearity are all equivalent.

\begin{remark}
The definition of LLP of a C*-algebra in \cite[Chp. 16 ]{pi} is different, it
requires completely contractive liftings from finite dimensional operator
subspaces. However, as it can be seen in \cite[Thm. 16.2]{pi}, all the
approaches coincide for C*-algebras. 
\end{remark}

\noindent \textbf{Note:} When we work with the finite dimensional operator
systems we remove the extra word ``local'', we even remove ``os'' and simply say
``lifting property''.

\subsection{Weak Expectation Property (WEP)}

\noindent If $\cl A$ is a unital C*-algebra then the bidual C*-algebra $\cl
A^{**}$ is unitally and completely order isomorphic to the bidual operator system
$\cl A^{dd}$. This allows us to extend the notion of weak expectation to a more general setting.
We say that an operator system $\cl S$ has
\textit{WEP}  if the canonical inclusion $i:\cl S \hookrightarrow \cl S^{dd}$
extends to a ucp map on the injective envelope $I(\cl S)$.
$$
\xymatrix{
 \cl S \ar@{^{(}->}[rr]^{i} \ar@{}[d]|-{\bigcap} &     & \cl S^{dd}  \\
 I(\cl S) \ar@{.>}[rru] &     &
}
$$
In \cite{kptt2} it was shown that  WEP implies (el,max)-nuclearity and the
difficult converse is shown in \cite{Han}. Consequently we have that
\begin{theorem}
An operator system has WEP if and only if it is (el,max)-nuclear.
\end{theorem}
\noindent If $\cl S$ is a finite dimensional operator system then $\cl S$
has WEP if and only if $\cl S$ has the structure of a C*-algebra (so it is the direct sum of matrix algebras).
This follows by the fact that the canonical injection $\cl S \hookrightarrow \cl S^{dd}$ is also surjective. Consequently, the \textit{expectation} (the extended
ucp map from $I(\cl S)$ into $\cl S^{dd}$) can be used to define a multiplication on $\cl S^{dd} \cong \cl S$ (see \cite[Theorem 15.2]{Pa}, e.g.).

\subsection{Double Commutant Expectation Property (DCEP)}\label{sec DCEP}

For a unital C*-algebra $\cl A$ the following are equivalent:
\begin{enumerate}
 \item The canonical inclusion $\cl A$ into $\cl A^{**}$ extends to cp map on $I(\cl A)$.
 \item For all (unital, C*-algebraic) inclusion $\cl A \subset B(H)$ there is a ucp map $\gamma : B(H) \rightarrow \cl A^{''}$,
where $\cl A^{''}$ is the double commutant of $\cl A$, extending the inclusion $\cl A \subset B(H)$.
\end{enumerate}
In either case we say that $\cl A$ has the WEP. As pointed out by Vern Paulsen it can be shown, 
by using Arveson's commutant lifting theorem (\cite{Ar1} or \cite[Thm. 12.7]{Pa}) e.g., that in (2) the C*-algebraic
inclusion can be replaced by unital complete order embedding. Moreover, by using straightforward injectivity
techniques, one can replace the domain of $\gamma$ by the injective envelope $I(\cl A)$ of $\cl A$. These two
equivalent notions differ in operator system setting. While (1) leads to the concept of WEP for general operator systems (above), 
(2) extends as follows: 

\begin{definition}
We say that $\cl S$ has
\textbf{DCEP} if every (operator system) inclusion $\cl S \subset B(H)$ extends
to a ucp map from $I(\cl S)$ into $\cl S''$, the double commutant of $\cl S$ in
$B(H)$.
$$\hspace{2cm}
\xymatrix{
 \cl S \ar@{^{(}->}[rr] \ar@{}[d]|-{\bigcap} &     &B(H) \supseteq  S''
\hspace{3cm}\\
 I(\cl S) \ar@{.>}[rru] &     &
}
$$
\end{definition}

\noindent Many fundamental results and problems regarding WEP for unital C*-algebras
reduce to DCEP for general operator systems. The following is a
direct consequence of Theorem 7.1 and 7.6 in  \cite{kptt2}:
\begin{theorem}
The following are equivalent for an operator system $\cl S$:
\begin{enumerate}
 \item  $\cl S$ is (el,c)-nuclear, that is, $\cl S \otimes_{el} \cl T = \cl S \otimes_{c} \cl T$ for every $\cl T$.
 \item $\cl S$ has DCEP.
 \item $\cl S \otimes_{min} C^*(\mathbb{F}_\infty) = \cl S \otimes_{max} C^*(\mathbb{F}_\infty)$.
 \item For any $\cl S \subset \cl A$ and $\cl B$, where $\cl A$ and $\cl B$ are
unital C*-algebras, we have the (operator system) embedding  $\cl S \otimes_{max} \cl B \subset \cl
A \otimes_{max} \cl B$.
\end{enumerate}
\end{theorem}

\noindent Here $C^*(\mathbb{F}_\infty)$ is the full group C*-algebra of the free group on countably infinite number of generators $\mathbb{F}_{\infty}$.
Note that (3) is Kirchberg's WEP characterization in \cite{Ki2} and (4) is
Lance's seminuclearity in \cite{La} for unital C*-algebras.

Note that the equivalence WEP and DCEP for unital C*-algebras can be inferred from the fact that $c$ and $max$ coincide
when one of the tensorant is a C*-algebra. So (el,max)-nuclearity (WEP) and (el,c)-nuclearity (DCEP)
coincide in this case. For general operator systems WEP $\Longrightarrow$ DCEP which simply
follows from  (el,max)-nuclearity implies  (el,c)-nuclearity. The converse fails even on two dimensional
operator systems: In \cite{Kav2} it was shown that every two dimensional operator system
is (min,c)-nuclear hence they have DCEP. But, by the last paragraph in the above subsection,
only  $\bb C^2$ (and its isomorphism class) has WEP.

\subsection{Exactness}\label{subsec exact} 
The importance of exactness and its connection to the tensor theory of
C*-algebras ensued by Kirchberg \cite{Ki1}, \cite{Ki2}. Exactness is a
categorical concept and requires a correct notion of quotient theory. The operator
system quotients established in \cite{kptt2} that we reviewed in preliminaries 
has been used to extend the exactness to general operator systems. Before
starting the definition we recall a couple of results from \cite{kptt2}: Let $\cl
S$ be an operator system, $\cl A$ be a unital C*-algebra and $I$ be an ideal in
$\cl A$. Then $\cl S \bar{\otimes} I $ is a kernel in $ \cl S 
\hat{\otimes}_{min} \cl A$ where  $\hat{\otimes}_{min}$ represents the completed
minimal tensor product and $\bar{\otimes}$ denotes the closure of the algebraic
tensor product in the larger space. By using the functoriality of the minimal
tensor product it is easy to see that the map
$$
 \cl S \hat{\otimes}_{min} \cl A \xrightarrow{id\otimes q} \cl S
\hat{\otimes}_{min} (\cl A / I),
$$
where $id$ is the identity on $\cl S$ and $q$ is the quotient map from $\cl A$
onto $\cl A / I$, is ucp and its kernel contains $\cl S \bar{\otimes} I $.
Consequently the induced map
$$
 (\cl S \hat{\otimes}_{min} \cl A) /  (\cl S \bar{\otimes} I ) \longrightarrow
\cl S \hat{\otimes}_{min} (\cl A / I)
$$
is still unital and completely positive. An operator system is said to be
\textbf{exact} if this induced map is a bijective and a complete order
isomorphism for every C*-algebra $\cl A$ and ideal $I$ in $\cl A$. In other
words we have the equality
$$
 (\cl S \hat{\otimes}_{min} \cl A) /  (\cl S \bar{\otimes} I ) = \cl S
\hat{\otimes}_{min} (\cl A / I).
$$
We remark that the induced map may fail to be surjective or injective, moreover
even if it has these properties it may fail to be a complete order isomorphism.

\begin{remark}\label{rem ex for fd} If $\cl S$ is finite dimensional then $\cl S \otimes_{min} \cl A = \cl S \hat{\otimes}_{min} \cl A$
and $\cl S \bar{\otimes} I = \cl S \otimes I $. Moreover the induced (ucp) map
$$
(\cl S \otimes_{min} \cl A) /  (\cl S \otimes I ) \longrightarrow
\cl S \otimes_{min} (\cl A / I)
$$
 is always bijective. Thus, for this case, exactness is equivalent to the inverse of induced map being a complete order isomorphism
for every $\cl A$ and $I \subset \cl A$ (see \cite{Kav2}).
\end{remark}

\noindent \textbf{Note:} The term exactness in this paper coincides with
1-exactness in \cite{kptt2}.

$ $

\noindent A unital C*-algebra is exact (in the sense of Kirchberg) if and only
if it is an exact operator system which follows from the fact that the unital C*-algebra ideal quotient coincides
with the operator system kernel quotient. The following is Theorem 5.7 of \cite{kptt2}:

\begin{theorem}\label{thm exact=(min,el)}
An operator system is exact if and only if it is (min,el)-nuclear.
\end{theorem}
In Theorem \ref{thm ex-dual-LP} we will see that  exactness and the lifting property
are dual pairs. We want to finish this subsection with the following stability property \cite{kptt2}:

\begin{proposition}\label{prop exac pass ss}
Exactness passes to operator subsystems. That is, if $\cl S$ is exact then every operator subsystem of $\cl S$
is exact. Conversely, if every finite dimensional operator subsystem of $\cl S$ is exact then $\cl S$ is exact.
\end{proposition}

\subsection{Final Remarks, Stability, and Examples on Nuclearity}
In this subsection we review the behavior of nuclearity related properties under basic algebraic operations such as
quotients and duality. We start with
the following nuclearity schema which summarizes the tensorial characterizations of several properties we have discussed:
$$\hspace{4cm}
\xymatrix{
\;\;\;\;\;\;\;\; min \;\;\; \leq \ar@{-}@/^1pc/[r]^{\;\;\;\;\;exactness}  
\ar@{-}@/^4pc/[rrr]^{\mbox{C*-nuclearity}} 
\ar@{-}@/_2pc/[rr]_{\mbox{osLLP}}  
  &
el    \ar@{-}@/_2pc/[rr]_{\mbox{DCEP}} 
   & er & \leq \;\;\;  c
\hspace{1cm}&  & &
}
\hspace{4cm}
\xymatrix{
\;\;\;\;\;\;\;\; min \;\;\; \leq \ar@{-}@/^1pc/[r]^{\;\;\;\;\;exactness}  
\ar@{-}@/^4pc/[rrr]^{\mbox{C*-nuclearity}} 
\ar@{-}@/_2pc/[rr]_{\mbox{osLLP}}  
  &
el    \ar@{-}@/_2pc/[rr]_{\mbox{DCEP}} 
   & er & \leq \;\;\;  c
\hspace{1cm}&  & &
}
$$
\noindent An operator system $\cl S$ is said to be \textbf{C*-nuclear} if $\cl S \otimes_{min} \cl A = \cl S \otimes_{max} \cl A$ for every C*-algebra
$\cl A$. It is elementary to show that (\cite{Kav2}) C*-nuclearity and (min,c)-nuclearity coincide.

A C*-algebra $\cl A$ is said to be \textbf{nuclear} if $\cl A \otimes_{min} \cl B = \cl A \otimes_{max} \cl B$ for every C*-algebra $\cl B$. It follows
that a unital C*-algebra $\cl A$ is nuclear if and only if it is (min,max)-nuclear operator systems. So in this case
$\cl A$ has the all the properties in the above schema. We also remark that Han and Paulsen \cite{HP}
prove that an operator system $\cl S$ is (min,max)-nuclear if and only if it has \textit{completely positive factorization property} (in the sense of \cite{Ki1}).
This extends a well known result of Choi and Effros on nuclear C*-algebras.

\begin{remark}
Let $\cl S$ be a finite dimensional operator system. Then $\cl S$ is (c,max)-nuclear 
if and only if it has the structure of a C*-algebra (\cite{Kav2}). Consequently, if $\cl S$ is
a non-C*-algebra then C*-nuclearity is the highest nuclearity that one should expect
(of course among $min\leq el,er \leq c \leq max$). This, on the finite dimensional case, puts the importance
of above mentioned properties.
\end{remark}

\begin{example}
It is evident from the above table that exactness and DCEP together are equivalent to C*-nuclearity. It was shown
in \cite{Kav2} that the operator subsystem
$$
\cl R = span\{I, E_{12},E_{21},E_{34},E_{43}\} \subset M_4
$$
does not have the lifting property. Clearly it is exact (since $M_4$ is nuclear, in particular, $M_4$ is exact so any operator subsystem is exact). Note that
$\cl R$ cannot have DCEP as DCEP + exactness = C*-nuclearity $ \Rightarrow $ LP. 
\end{example}

\begin{example}\label{ex WEPLLP}
Suppose $\cl A$ and $\cl B$ are unital C*-algebras such that $\cl A$ has LLP (eq. osLLP) and $\cl B$ has WEP (eq. DCEP). Then, by using
the tensor characterizations in the above table, we have
$$
\cl A \otimes_{min} \cl B = \cl A \otimes_{er} \cl B.
$$
Since $\cl B$ has WEP, (and taking into account the asymmetry of $el$ and $er$) we have
$$
 \cl A \otimes_{er} \cl B = \cl A \otimes_{c = max} \cl B.
$$
Thus $\cl A \otimes_{min} \cl B = \cl A \otimes_{max} \cl B$, so we recover a well known result of Kirchberg \cite{Ki2}.
\end{example}

The relation of exactness and the lifting property perhaps best seen in the finite dimensional case.
Following is from \cite{Kav2}.

\begin{theorem}\label{thm ex-dual-LP}
Let $\cl S$ be a finite dimensional operator system. Then $\cl S$ is exact if and only if $\cl S^d$ has the lifting
property (and vice versa). In other words, $\cl S$ is (min,el)-nuclear if and only if $\cl S^d$ is (min,er)-nuclear. 
\end{theorem}

\begin{example}
Let $\cl R$ be the above operator system. Then the dual operator system $\cl R^d$ is not exact but
has the lifting property. We don't know whether $\cl R^d$ has DCEP. In \cite{Kav2} it was shown that the Kirchberg
conjecture has an affirmative answer if and only if $\cl R^d$ has DCEP.
\end{example}

In contrast to C*-algebra ideal quotients the lifting property is stable quotients by null-subspaces \cite{Kav2}:

\begin{theorem}\label{thm QouhasLP}
Let $\cl S$ be a finite dimensional operator system and $J$ be a null-subspace in $\cl S$. If $\cl S$ has the
lifting property then $\cl S / J$ has the same property.
\end{theorem}

\begin{example}
Let $J_3 \subset M_3$ be the diagonal operators with 0-trace. Then $M_3 / J_3$ has the lifting property. This follows
by the fact that $M_3$ is nuclear so has the lifting property.
\end{example}

\begin{example}
Unlike to (separable) C*-algebra/ideal quotients exactness is not preserved under operator system/kernel quotients.
It was shown in \cite{Kav2} that $M_3 / J_3$ is not exact (depending heavily on a result of Wassermann \cite{Was2}). 
We don't know whether $M_3 / J_3$ has DCEP or not. This problem
is again equivalent to the Kirchberg conjecture.
\end{example}

We will also need the following fact from \cite{Han} and \cite{Kav2}. First we remark that given operator systems  $\cl S$ and $\cl T$
if $dim(T)$ is finite then the completion $\cl S \hat \otimes_{\tau} \cl T$  of any tensor product $\tau$ is
same as the algebraic tensor product.

\begin{theorem}\label{thm max-quotient}
Let $\cl S$ and $\cl T$ be operator systems with $dim(T) < \infty$ and let $J$ be a null-subspace of $\cl T$. 
Then we have the unital complete order isomorphism
$$
(\cl S \otimes_{max} \cl T) / (\cl S \otimes J) = \cl S  \otimes_{max} (\cl T/J).
$$
If, in addition,   $dim(S)< \infty$, then $\cl S \otimes J $ is a null-subspace of $ \cl S \otimes_{max} \cl T$ (and so a completely
proximinal kernel).
\end{theorem}

\section{Coproducts of Operator Systems}

In this section we review the the amalgamated sum of two operator systems over their
unit introduced in \cite{KL} (or coproduct of two operator systems with the language
of \cite{TF}). Given two operator systems
$\cl S$ and $\cl T$ there is an operator system $\cl U$ with unital
and completely order embeddings $i: \cl S \hookrightarrow \cl U$
and $j: \cl T \hookrightarrow \cl U$ such that the following holds: Whenever
$\phi: \cl S \rightarrow \cl R$ and $\psi: \cl T \rightarrow \cl R$
are ucp maps, where $\cl R$ is any operator system, then
there is a unique ucp map $\varphi: \cl U \rightarrow \cl R$ such that
$\phi(s) = \varphi (i (s))$ and $\psi(t) = \varphi ( j (t))$ for all $s \in \cl S$
and $t \in \cl T$. We will call $\cl U$ together with embeddings $i$ and $j$ the \textit{coproduct}
of $\cl S$ and $\cl T$ and we will denote it by $\cl S \oplus_1 \cl T$. The following commuting
diagram summarizes the universal property of the coproduct.
$$
\xymatrix{
\cl S \ar@{_{(}->}[d]_i   \ar[drr]^{ucp \;\phi} && \\
\cl S   \oplus_1 \cl T  \ar@{.>}[rr]^{! \; ucp \; \varphi}  & & \cl R  \\
\cl T \ar@{^{(}->}[u]^j   \ar[urr]_{ucp\;\psi} & &
}
$$
Once such an object is proven to be exists it is easy to see that it has to be unique
up to a unital complete order isomorphism. We leave the verification of this
to the reader which is based on the fact that  $\cl U$ must be spanned
by the elements of $i(\cl S)$ and $j(\cl T)$ as $\varphi$ is uniquely determined.

There are several ways to construct the coproduct of two operator systems. We first recall
the  free product of C*-algebras. Given unital C*-algebras $\cl A$ and $\cl B$
the \textit{unital free product} $\cl A *_1 \cl B$ is a C*-algebra with C*-algebraic inclusions
$i: \cl A \rightarrow \cl A *_1 \cl B$ and $j: \cl B \rightarrow \cl A *_1 \cl B$ such that
whenever $\pi: \cl A \rightarrow \cl C$ and $\rho:  \cl B \rightarrow \cl C$ are
unital $*$-homomorphisms, where $\cl C$ is any C*-algebra, then there is a unique unital
$*$-homomorphism $\gamma: \cl A *_1 \cl B \rightarrow \cl C$ such that $\gamma(i(a)) = \pi(a)$
and $\gamma(j(b)) = \rho(b)$ for all $a \in \cl A$ and $b \in \cl B$. We often identify $\cl A$
and $\cl B$ with their canonical images in $\cl A *_1 \cl B$. First and the easiest way to see the existence
of the coproduct is the following.

\begin{proposition}
Let $\cl S$ and $\cl T$ be two operator systems. Then
$$
\cl S \oplus_1 \cl T = \{s + t : s \in \cl S, \; t \in \cl T   \} \subset C^*_u(\cl S) *_1C^*_u(\cl T). 
$$
\end{proposition}

\begin{proof}
Suppose $\phi: \cl S \rightarrow \cl R$ and $\psi: \cl T \rightarrow \cl R$
are ucp maps, where $\cl R$ is any operator system. Let $\cl R \subset B(H)$. Both $\phi$
and $\psi$ extends to unital $*$-homomorphisms $\pi:C^*_u(\cl S) \rightarrow B(H)$
and $\rho:  C^*_u(\cl T) \rightarrow B(H)$, respectively. Let $\gamma$ be the unital
$*$-homomorphism from  $C^*_u(\cl S) *_1C^*_u(\cl T) $ into $B(H)$ extending $\pi$ and $\rho$.
Now by restricting $\gamma$ the operator subsystem  $\{s+t : s \in \cl S, \; t \in \cl T   \} $ we obtain a ucp map. Note that
the image of the restricted $\gamma$ still lies in $\cl R$.
\end{proof}

A more general form of this will be more useful. We will need the following:

\begin{lemma}
Let $\cl A$ and $\cl B$ be unital C*-algebras such that they both have a one dimensional representation,
that is, there are unital $*$-homomorphisms $w_1 :\cl A \rightarrow \mathbb{C}$
and $w_2 : \cl B \rightarrow \mathbb{C}$. If $\phi: \cl A \rightarrow B(H)$ and
$\psi: \cl B \rightarrow B(H)$ are ucp maps then there is a ucp map $\varphi:  \cl A *_1 \cl B \rightarrow B(H)$
such that $\varphi(a) = \phi(a)$ and $\varphi(b) = \psi(b)$ for every $a \in \cl A$ and  $b \in \cl B$.
\end{lemma}
\begin{proof}
By using the Stinespring representation theorem we can find a Hilbert space $H_1$ and a unital $*$-homomorphism
$\pi: \cl A \rightarrow B(H \oplus H_1)$ such that $\phi = V^* \pi(\cdot) V$ where $V^* = (1\;0)$. Similarly let 
$(\rho, H\oplus H_2,W)$ be the Stinespring representation of $\psi$. Set $K = H \oplus H_1 \oplus H_2$. Let
$\tilde{\pi} = \pi \oplus w_1(\cdot) I_2$ and similarly let $\tilde{\rho} = \rho \oplus  w_2(\cdot) I_1$. More precisely,
if
$$
\pi(a) =
\left(
\begin{array}{cc}
\phi(a) & x \\
y & z
\end{array}
\right)
\mbox{ and }
\rho(b) =
\left(
\begin{array}{cc}
\psi(a) & X \\
Y & Z
\end{array}
\right)
$$
for some $a \in \cl A$ and $b\in \cl B$ written in the matrix from w.r.t.\! $H\oplus H_1$ and $H\oplus H_2$, respectively, then
$$
\tilde{\pi}(a) =
\left(
\begin{array}{ccc}
 \phi(a) & x & 0\\
y & z & 0 \\
0 &0 &w_1(a)I_2
\end{array}
\right)
\mbox{ and }
\tilde{\rho}(b) =
\left(
\begin{array}{ccc}
 \phi(a) & 0 & X\\
0 & w_2(b)I_1 & 0 \\
Y &0 & Z
\end{array}
\right).
$$
Clearly $\tilde{\pi}$ and $\tilde{\rho}$ are still unital $*$-homomorphisms such that $\phi  = U^* \tilde{\pi}(\cdot) U$ and
$\psi  = U^* \tilde{\rho}(\cdot) U$ where $U^* = (1\;0\;0)$. Let $\gamma: \cl A *_1 \cl B \rightarrow B(K)$ be the
unital $*$-homomorphism extending $\tilde{\pi}$ and $\tilde{\rho}$. Clearly $\varphi = U^* \gamma(\cdot) U$ is a
ucp map from $\cl A *_1 \cl B$ into $B(H)$ with desired properties.
\end{proof}

Following \cite{kptt2}, we will say that an operator subsystem $\cl S$ of a unital C*-algebra $\cl A$ \textit{contain enough unitaries} if
there is a collection of unitaries that belongs to $\cl S$ and generates $\cl A$ as a C*-algebra, that is,
$\cl A$ is the smallest C*-algebra that contains these unitaries.

\begin{proposition}
Let $\cl A$ and $\cl B$
be unital C*-algebras such that they both have a one dimensional representation.
Let $\cl S \subset \cl A$
and $\cl T \subset \cl B$ be operator subsystems. Then we have
$$
\cl S \oplus_1 \cl T = \{ s+t: s \in \cl S, \; t \in \cl T   \} \subset \cl A *_1 \cl B.
$$
Moreover, if $\cl S$ and $\cl T$ contains enough unitaries in $\cl A$ and $\cl B$, resp., then 
$\{ s+t: s \in \cl S, \; t \in \cl T   \} $ contains enough unitaries in $\cl A *_1 \cl B$.
\end{proposition}
\begin{proof}
We simply show that the operator system $\{ s + t: s \in \cl S, \; t \in \cl T   \} \subset \cl A *_1 \cl B$ satisfies the universal property
of the coproduct of $\cl S$ and $\cl T$. Let $\phi: \cl S \rightarrow \cl R$ and $\psi: \cl T \rightarrow \cl R$
be two ucp maps, where $\cl R$ is any operator system. Let $\cl R \subset B(H)$. By using Arveson's extension
theorem, let $\tilde{\phi}: \cl A \rightarrow B(H)$ and 
$\tilde{\psi}: \cl B \rightarrow B(H)$ be the ucp extensions of $\phi$ and $\psi$. By the above lemma there is a ucp map
 $\varphi:  \cl A *_1 \cl B \rightarrow B(H)$
such that $\varphi(a) = \tilde{\phi}(a)$ and $\varphi(b) =\tilde{\phi}(b)$ for every $a \in \cl A$ and  $b \in \cl B$. Now the restriction
of the this map on $\{ s + t: s \in \cl S, \; t \in \cl T   \}$ is the desired extension of $\phi$ and $\psi$. To see the final part
note  that $\cl A \subset \cl A*_1 \cl B$ as a C*-subalgebra. So unitaries in $\cl S$ already generates $\cl A \subset \cl A*_1 \cl B$.
Likewise, unitaries in $\cl T$ generates $\cl B\subset \cl A*_1 \cl B$. Since $\cl A*_1 \cl B$ is the smallest C*-algebra
that contains $\cl A$ and $\cl B$ the result follows.
\end{proof}

\begin{corollary}\label{cor AoplusB}
Let $\cl A$ and $\cl B$ be two unital C*-algebras. Suppose both $\cl A$ and $\cl B$ have one dimensional
representations. Then
$$
\cl A \oplus_1 \cl B = \{a+b: a\in \cl A,\; b\in \cl B \} \subset \cl A*_1 \cl B.
$$
Moreover, the operator subsystem $\{a+b: a\in \cl A,\; b\in \cl B \} $ contains enough unitaries in $ \cl A*_1 \cl B$.
\end{corollary}

\begin{remark}
For any operator system $\cl S$, $C^*_u(\cl S)$ possesses a one dimensional representation.
More precisely, if $f: \cl S \rightarrow \mathbb{C}$ is a state then it extends to unital $*$-homomorphism
on  $C^*_u(\cl S)$. Thus, our first construction of coproduct is a special case of the above proposition.
\end{remark}

\begin{remark}
If $G$ is a discrete group then the full C*-algebra $C^*(G)$ has a one dimensional representation.
In fact, $\rho: G \rightarrow \mathbb{C}$ given by  $\rho(g) = 1$ is a unitary representation so the universal
property of $C^*(G)$ ensures that $\rho$ extends to a unital $*$-homomorphism on $C^*(G)$.
\end{remark}

Now, we will discuss a more concrete construction given in \cite{Kav2}. This also justifies the notation
of the coproduct.
 If $\cl S$ and $\cl T$ be operator systems
then $span\{(e,-e)\}$ is a one dimensional null subspace of $\cl S \oplus \cl T$. We will show that:

\begin{proposition}\label{pro SoplusT}
$ \cl S \oplus_1 \cl T = (\cl S \oplus \cl T)/span\{(e,-e)\}$.
\end{proposition}

\begin{proof}
Consider $i: \cl S \rightarrow (\cl S \oplus \cl T)/span\{(e,-e)\}$ given by $s \mapsto (2s, 0)+J$.
We claim that $i$ is a unital complete order embedding. First note that
$$
i(e) = (2e,0)+J = (2e,0)+(-e,e) + J = (e,e)+J
$$
so $i$ is unital. It is also easy to see that $i$ is a cp map as it can be written as the composition of two cp maps:
$\cl S \rightarrow \cl S \oplus \cl T$, $s\mapsto (2s,0)$
and the quotient map. To see that it is a complete order embedding let $((2s_{ij},0)+J)$ be positive
in $M_k((\cl S \oplus \cl T)/span\{(e,-e)\})$. We will show that $(s_{ij})$ is in $M_k(\cl S)^+$. Since
quotient by null subspaces are completely proximinal $((2s_{ij},0)+J)$ has positive
representative in $\cl S \oplus T$, say
$$
(2s_{ij},0) + (\alpha_{ij}e,-\alpha_{ij}e)  = (2s_{ij} + \alpha_{ij}e,- \alpha_{ij}e).
$$
This clearly forces $( - \alpha_{ij}e)$ to be positive in $M_k(\cl T)$, equivalently,
we have that $( - \alpha_{ij}) \in M_k^+$. Finally,
Since $(2s_{ij} + \alpha_{ij}e)$ and $(- \alpha_{ij}e)$ are two positive elements of $M_k(\cl S)$
it follows that their addition is also positive. This proves that $i$ is a unital complete order embedding.

\noindent Likewise $j: \cl T \rightarrow (\cl S \oplus \cl T)/span\{(e,-e)\}$; $t \mapsto (0, 2t)+J$ is a unital complete
order embedding. It is also well establish that

\noindent Finally suppose $\phi: \cl S \rightarrow \cl R$ and $\psi: \cl T \rightarrow \cl R$
are two ucp maps, where $\cl R$ is an operator system. Define $\varphi :  (\cl S \oplus \cl T)/span\{(e,-e)\} \rightarrow \cl R$
by $\varphi((s,t) + J) = \phi(s)/2 + \psi(t)/2$. Note that $\varphi$ is well defined as both $\phi$ and $\psi$ are unital.
Moreover $\varphi(i(s)) = \varphi ((2s,0)+J) = \phi(s) $. Likewise $\varphi(j(t)) = \psi(t)$. Thus we only need to
show that $\varphi$ is a cp map. Consider $\gamma: \cl S \oplus \cl T \rightarrow \cl R$ given by $\gamma((s,t)) = \phi(s)/2 + \psi(t)/2$.
Clearly $\gamma$ is a ucp map such that $(e,-e)$ belongs to its kernel. Consequently the induced map $\bar{\gamma}: 
(\cl S \oplus \cl T)/span\{(e,-e)\} \rightarrow \cl R $ is still a ucp map by the universal property of the operator system quotients.
Clearly $\bar{\gamma}$ coincides with $\varphi$. This finishes the proof. 
\end{proof}

Supposing $G$ and $H$ discrete groups then we have the C*-algebraic isomorphism
$$
C^*(G) *_1 C^*(H) \cong C^*(G * H)
$$
in a natural way. We refer the reader to \cite[pg. 149]{pi} on a discussion on this topic. 
Letting $\bb Z_k$ be the cyclic group of order $k$, it is well known  that the group C*-algebra $C^*(\bb Z_k)$
can be identified with $\bb C^k$ (see \cite[pg. 60]{BEF}). Consequently we have that
\begin{eqnarray}\label{rrrr}
 \bb C^k *_1 \bb C^m \cong C^*(\bb Z_k) *_1 C^*(\bb Z_m)  \cong C^*(\bb Z_k * \bb Z_m).
\end{eqnarray}
Note that the group $\bb Z_k * \bb Z_m$ can be given as $\bb Z_k * \bb Z_m = \langle a,b : a^k = b^m = e \rangle$. By identifying $\bb Z_k * \bb Z_m$
with its canonical image in $C^*(\bb Z_k * \bb Z_m)$ we have that
\begin{eqnarray}\label{r}
\cl S = span\{e,a,a^2,...,a^{k-1},b,b^2,...,b^{m-1}\}  \subset C^*(\bb Z_k * \bb Z_m)
\end{eqnarray}
is closed under the involution and consequently an operator subsystem. Under the natural identification between the C*-algebras
given in \ref{rrrr} the following operator subsystems
$$
\{x+y: \; x\in \bb C^k,\; x\in \bb C^k \} \subset \bb C^k *_1 \bb C^m \mbox{ and } \cl S \subset C^*(\bb Z_k * \bb Z_m)
$$
are invariant and consequently they are unitally completely order isomorphic. This allows us to draw 
the following conclusion which is the main purpose of this section.  

\begin{theorem}\label{thm equiv}
Let $k$ and $m$ be positive integers. The following operator systems are unitally completely order isomorphic:
\begin{enumerate}
 \item $\bb C^{k+m}/span\{(\underbrace{1,...,1}_{k-terms},\underbrace{-1,...,-1}_{m-terms})\}$.

\item $\bb C^k \oplus_1 \bb C^k$.

\item $\{x+y: x \in \bb C^k, y\in \bb C^m\} \subset \bb C^k *_1 \bb C^m $.

\item The operator subsystem $\cl S \subset C^*(\bb Z_k * \bb Z_m)$ given in \ref{r}.
\end{enumerate}
Moreover, each of these operator systems contain enough unitaries in $C^*(\bb Z_k * \bb Z_m)$.
\end{theorem}
\begin{proof}
(1)$=$(2) can be seen by Proposition \ref{pro SoplusT} and (2)$=$(3)
follows from Corollary \ref{cor AoplusB}. By the above discussion we have (3)$=$(4).
Finally, clearly the operator system in (4)  contains enough unitaries in $C^*(\bb Z_k * \bb Z_m)$.
\end{proof}

\begin{remark}\label{rem C5/J}
In the next section we are particularly interested in the four dimensional operator system $\bb C^5/span\{(1,1,-1,-1,-1)\}$
and exhibit several universal properties of this operator system. By the using the above theorem for $k = 2$, $m = 3$ we see that
\begin{enumerate}
 \item $\bb C^5/span\{(1,1,-1,-1,-1)\}$,
 \item $\bb C^2 \oplus_1 \bb C^3$,
 \item $\{x+y: x \in \bb C^2, y\in \bb C^3\} \subset \bb C^2 *_1 \bb C^3 $,
 \item $span\{e,a,b,b^*\} \subset C^*(\bb Z_2 * \bb Z_3)$ where $\bb Z_2 * \bb Z_3$ is given by $\langle a,b : a^2 = b^3 = e\rangle$
\end{enumerate}
are all unitally and completely order isomorphic. In \cite{Kav}, the operator system in (4) is given as an example of a four
dimensional operator system which is not exact. We will recover this in later sections.
\end{remark}

\section{Weak Expectation Property (WEP)}

In this section we exhibit a new WEP criteria and express the Kirchberg Conjecture in terms of a problem
about four dimensional operator system problem. As we pointed out in the introductory section WEP is
one of the fundamental nuclearity property ensued by Lance \cite{La}. It characterizes the \textit{semi-nuclear}
C*-algebras  in the sense that the maximal tensor product which is projective behaves injectively for
this class of C*-algebras. More precisely $\cl A$ has WEP if and only if for every C*-algebras  
$\cl B $ and $\cl C$ with $\cl A \subset \cl B$ one has
$$
\cl A \otimes_{max} \cl C \subset \cl B \otimes_{max} \cl C.
$$ 

Recall from last section that an operator subsystem $\cl S$ of a C*-algebra $\cl A$ is said to \textit{contain enough
unitaries} if there is a collection of unitaries in $\cl S$ that generates $\cl A$ as a C*-algebra. Such an
operator system contains great deal of information about the nuclearity properties of $\cl A$ which we shall state below.
In many instances this allows us to retrieve the properties of a C*-algebra by using a very low dimensional operator system.
The following proposition is the key point and it was inspired by a work of Pisier \cite{pi}. The proof can be found in \cite{kptt2}.
\begin{proposition}\label{pro enough}
Suppose $\cl S \subset \cl A$ and $\cl T \subset \cl B$ contain enough unitaries. Then
$$
\cl S \otimes_{min} \cl T = \cl S \otimes_{c} \cl T \;\; \Longrightarrow \;\; 
\cl A \otimes_{min} \cl B = \cl A \otimes_{max} \cl B.
$$
\end{proposition}

\begin{proposition}
Let $\cl S \subset \cl A$ contain enough unitaries. Then:
\begin{enumerate}
\item $\cl S$ is exact $\Longrightarrow $ $\cl A$ is exact.
\item $\cl S$ has DCEP $\Longrightarrow $ $\cl A$ has WEP.
\item $\cl S$ is C*-nuclear $\Longrightarrow $ $\cl A$ nuclear.
\item $\cl S$ has osLLP $\Longrightarrow $ $\cl A$ has LLP.
\end{enumerate}
\end{proposition}
\begin{proof}
The reader may refer to \cite{kptt2} for the proofs of (1), (2), and (3). We only prove (4). By Theorem \ref{thm LLPchar}, 
the condition on $\cl S$ requires that $\cl S \otimes_{min} B(H) = \cl S \otimes_{max} B(H)$ for every Hilbert space $H$.
So by the above proposition (with $\cl T = B(H) = \cl B$ and
using the fact that $c$ and $max$ coincides when one of the tensorant is a C*-algebra) we get
$\cl A \otimes_{min} B(H) = \cl A \otimes_{max} B(H)$ for every $H$, equivalently,  $\cl A$ has LLP. 
\end{proof}

In \cite{bo}, Boca proves that LLP is preserved under unital free products. Consequently:
\begin{corollary}[Boca]
The group C*-algebra $C^*(\bb Z_k*\bb Z_m)$ has LLP.
\end{corollary}
 \begin{proof}
 This is a consequence of the identification
$$
 \bb C^k *_1 \bb C^m \cong C^*(\bb Z_k) *_1 C^*(\bb Z_m)  \cong C^*(\bb Z_k * \bb Z_m).
$$
Since $\bb C^k$ and $\bb C^m$ has LLP it follows that  $C^*(\bb Z_k*\bb Z_m)$ has LLP.
 \end{proof}

\begin{rem}
This can be alternately proved as follows: Since $\bb C^{k+m} / J$,
where
$$
J = span\{(\underbrace{1,...,1}_{k-terms},\underbrace{-1,...,-1}_{m-terms})\},
$$
can be identified with an operator subsystem of $C^*(\bb Z_k*\bb Z_m)$ which contain enough unitaries, by the above
proposition, it is enough to prove that $\bb C^{k+m} / J$ has the lifting property. But this a simple consequence of 
Theorem \ref{thm QouhasLP}.
\end{rem}

Let $\bb F_{\infty}$ be the free group on the countably infinite number of generators and let
$C^*(\bb F_\infty)$ be the full group C*-algebra of $\bb F_{\infty}$. 
It is well establish that the free group $\bb F_{\infty}$ embeds in $\bb Z_2 * \bb Z_3$ (see \cite[Pg. 24]{PDLH} e.g.). In the following
we identify the groups with their canonical images in the their (full) group C*-algebras. It is essentially \cite[Prop. 8.8]{pi}.

\begin{proposition}
Let $H$ be subgroup of $G$. Then $C^*(H)$ embeds in $C^*(G)$. More precisely, the unitary
representation $\rho: H \rightarrow C^*(G)$ given by  $h\mapsto h$ extends to bijective unital $*$-homo\-morphism $\pi$.
Moreover, this embedding has a ucp inverse.
\end{proposition}

So roughly speaking if $H$ is a subgroup of $G$ then the identity on $C^*(H)$ decomposes via ucp maps on $C^*(G)$.
The following is a direct consequence of  \cite[Lem. 5.2]{Kav2}.

\begin{proposition}
Suppose $\cl A$ and $\cl B$ are C*-algebras such that the identity decomposes via ucp map on $\cl B$, that is,
there are ucp maps $\phi: \cl A \rightarrow \cl B$ and $\psi: \cl B \rightarrow \cl A$ such that $\psi(\phi(a)) = a$ for all $a$ in $\cl A$.
Then any nuclearity property of $\cl B$ passes to $\cl A$. More precisely,
\begin{enumerate}
 \item if $\cl B$ is nuclear then $\cl A$ is nuclear;
  \item if $\cl B$ is exact then $\cl A$ is exact;
  \item if $\cl B$ has WEP then $\cl A$ has WEP;
 \item if $\cl B$ has LLP then $\cl A$ has LLP.
\end{enumerate}
\end{proposition}

Since $\bb F_{\infty}$ embeds in $\bb Z_2 * \bb Z_3$, the full C*-algebra $C^*(\bb F_{\infty})$ embeds
in $C^*(\bb Z_2 * \bb Z_3)$ with a ucp inverse. So, by the above proposition, any nuclearity property
of $C^*(\bb Z_2 * \bb Z_3)$ passes to $C^*(\bb F_{\infty})$. We are now ready to state Kirchberg's
WEP characterization \cite{Ki2} and its slight modification which will be more useful for us:
\begin{theorem}\label{thm F_2Z_2}
Let $\cl A$ be a unital C*-algebra. Then the following are equivalent:
\begin{enumerate}
 \item $\cl A$ has WEP.
 \item $\cl A \otimes_{min} C^*(\bb F_{\infty}) = \cl A \otimes_{max} C^*(\bb F_{\infty})$.
 \item $\cl A \otimes_{min} C^*(\bb Z_2 * \bb Z_3) = \cl A \otimes_{max} C^*(\bb Z_2 * \bb Z_3)$.
\end{enumerate}
\end{theorem}

\begin{proof}
The equivalence of (1)  and (2) is Kirchberg's WEP characterization. (1) implies (3) follows from
the fact that $C^*(\bb Z_2 * \bb Z_3)$ has LLP. Since $\cl A$ has WEP, by Example \ref{ex WEPLLP}, we obtain (3). 
To see (3) $\Rightarrow$ (2), let $\phi: C^*(\bb F_{\infty}) \rightarrow C^*(\bb Z_2 * \bb Z_3)$ and 
$\psi: C^*(\bb Z_2 * \bb Z_3) \rightarrow C^*(\bb F_{\infty})$ bu ucp maps such that their composition is the
identity on  $ C^*(\bb F_{\infty})$. In the following we will use the fact that the \textit{min} is and the \textit{max} are functorial
tensor product. We will also use the fact that if the composition of two ucp maps is a complete order embedding
then the first map has the same property. We have
$$
\cl A \otimes_{max} C^*(\bb F_{\infty}) \xrightarrow{id \otimes \phi}  \cl A \otimes_{max} C^*(\bb Z_2 * \bb Z_3) \xrightarrow{id \otimes \psi}
\cl A \otimes_{max} C^*(\bb F_{\infty})
$$
is a sequence of ucp maps such that the composition is the identity on $\cl A \otimes_{max} C^*(\bb F_{\infty})$. This
means that $id \otimes \phi$ is a complete order embedding. If we consider the same ucp maps with max is replaced
by min we again see that $id \otimes \phi$ is a complete order embedding. Thus we have the embeddings
$$
\cl A \otimes_{max} C^*(\bb F_{\infty}) \xrightarrow{id \otimes \phi}  \cl A \otimes_{max} C^*(\bb Z_2 * \bb Z_3) \mbox{ and }
\cl A \otimes_{min} C^*(\bb F_{\infty}) \xrightarrow{id \otimes \phi}  \cl A \otimes_{min} C^*(\bb Z_2 * \bb Z_3).
$$
Since the tensor products on the right hand side coincides it follows that (3) implies (2).
\end{proof}

The following is one of our main results in this section. It is a four dimensional operator system version
of the above theorem. In the remaining of this and the next section we will have  several application of this
theorem. As usual $J$ stands for the one dimensional null-subspace $span\{(1,1,-1,-1,-1)\} \subset \bb C^5$.
\begin{theorem}\label{thm WEPcri}
A unital C*-algebra $\cl A$ has WEP if and only if 
$$
\cl A \otimes_{min} (\bb C^5/J) = \cl A \otimes_{max}  (\bb C^5/J).
$$
\end{theorem}

\begin{proof}
If $\cl A$ has WEP then it is (el,max)-nuclear. So we get 
$\cl A \otimes_{el} \cq{5} = \cl A \otimes_{max} \cq{5}$. Also, by Theorem \ref{thm QouhasLP}, 
$\cq{5}$ has the lifting property or (min,er)-nuclearity. Since it is written to right
hand side, by using the asymmetry of \textit{el} and \textit{er},  we have $\cl A \otimes_{min} \cq{5} = \cl A \otimes_{el} \cq{5}$. This
proves one direction. Now suppose the converse. From the last section we have that $\cq{5}$
unitally completely order isomorphic to an operator subsystem of $C^*(\bb Z_2 * \bb Z_3)$
which contains enough unitaries. Now, by using Proposition \ref{pro enough}, $\cl A \otimes_{min} \cq{5} = \cl A \otimes_{c=max} \cq{5}$
implies that $\cl A \otimes_{min} C^*(\bb Z_2 * \bb Z_3) = \cl A \otimes_{max} C^*(\bb Z_2 * \bb Z_3)$. Thus, by the
above theorem, $\cl A$ has WEP.
\end{proof}

The above theorem can also be extended to general operator systems. We need a preliminary lemma. The following is
\cite[Lem. 5.4]{Kav2}.
\begin{lemma}
Let $\cl T \subset \cl B$ contains enough unitaries, say $\{u_\alpha\}$, and $\varphi: \cl B \rightarrow \cl C$, where $\cl C$
is a C*-algebra, such that $\varphi(u_\alpha)$ is a unitary  in $\cl C$ for all $\alpha$ then $\varphi$ must be a unital
$*$-homomorphism.
\end{lemma}

\begin{theorem}\label{thm DCEPCri}
An operator system $\cl S$ has DCEP if and only if $\cl S \otimes_{min} (\bb C^5/J) = \cl S \otimes_{c}  (\bb C^5/J).$
\end{theorem}
\begin{proof}
First assume that $\cl S$ has DCEP (eq. (el,c)-nuclearity). Since $(\bb C^5/J)$ has the lifting property (or (min,er)-nuclearity)
we obtain
$$
\cl S \otimes_{min} (\bb C^5/J) = \cl S \otimes_{el}  (\bb C^5/J) = \cl S \otimes_{c} (\bb C^5/J).
$$
This proves one direction. Conversely let $\cl S \otimes_{min} (\bb C^5/J) = \cl S \otimes_{c}  (\bb C^5/J)$. Recall
that if we let $\bb Z_2 * \bb Z_3 = \langle a,b: a^2 = b^3 = e \rangle$ then $\cl R = span\{e,a,b,b^*\}$ is an operator
subsystem of $C^*(\bb Z_2 * \bb Z_3) $ which contains enough unitaries. Moreover $\bb C^5/J$ and $\cl R$ are unitally
completely order isomorphic. So our assumption is equivalent to $\cl S \otimes_{min} \cl R = \cl S \otimes_{c}  \cl R$.
As a first step we claim that $\cl S \otimes_{min} \cl R \subset \cl S \otimes_{max} C^*(\bb Z_2 * \bb Z_3)$ (and let $i$ denote this inclusion).
In fact letting $\tau$ be the operator system structure on $\cl S \otimes \cl R$ arising from the inclusion
$\cl S \otimes_{max} C^*(\bb Z_2 * \bb Z_3)$ we see that $min \leq \tau \leq c$. Since $min = c$ by our assumption
our claim follows. Secondly we wish to show that 
\begin{eqnarray}\label{111}
\cl S \otimes_{min} C^*(\bb Z_2 * \bb Z_3) = \cl S \otimes_{max} C^*(\bb Z_2 * \bb Z_3).
\end{eqnarray}
To see this we first represent $ \cl S \otimes_{max} C^*(\bb Z_2 * \bb Z_3)$ into a $B(H)$ such a 
way that the portions ``$ \cl S$'' and ``$C^*(\bb Z_2 * \bb Z_3)$'' commute and ``$C^*(\bb Z_2 * \bb Z_3)$'' is a
C*-subalgebra of $B(H)$. Let $\cl A$ be any C*-algebra containing $\cl S$ as an operator subsystem.
By the injectivity of $min$ we have $\cl S \otimes_{min} \cl R  \subset \cl A \otimes_{min} C^*(\bb Z_2 * \bb Z_3) $.
By using Arveson's lifting theorem we obtain a ucp map $\gamma :  \cl A \otimes_{min} C^*(\bb Z_2 * \bb Z_3) \rightarrow B(H)$
extending $i$.
$$
\xymatrix{
\cl S \otimes_{min} \cl R \ar@{^{(}->}[rr]^i   \ar@{}[d]|-{\bigcap}  & &\cl S \otimes_{max}C^*(\bb Z_2 * \bb Z_3)  \ar@{}[r]|-{\subset} & B(H) \\
 \cl A \otimes_{min} C^*(\bb Z_2 * \bb Z_3)  \ar@{.>}[rrru]_{\gamma} & & &}
$$
When $\gamma$ is restricted to ``$C^*(\bb Z_2 * \bb Z_3)$'' it must be the identity. In fact the above lemma ensures that it has 
to be a unital $*$-homomorphism as it maps $a,b,b^*$ to $a,b,b^*$, respectively. From this it is easy to see that it is the identity on
``$C^*(\bb Z_2 * \bb Z_3)$'' as $a,b,b^*$ generates ``$C^*(\bb Z_2 * \bb Z_3)$''. 
This means that $\gamma$ has to be $C^*(\bb Z_2 * \bb Z_3)$-module map
in the sense that for $x\in \cl A$ and $y \in C^*(\bb Z_2 * \bb Z_3)$ $\gamma(x \otimes y) = \gamma(x \otimes e)\gamma(e\otimes y)$.
This follows from the theory of Choi on multiplicative domains \cite[Thm. 3.18]{Pa}. Now if we restrict
$\gamma$ on $\cl S \otimes_{min} C^*(\bb Z_2 * \bb Z_3) $ it is the identity since for $s \in \cl S$ and $y \in C^*(\bb Z_2 * \bb Z_3)$
we have $\gamma(s \otimes y) = \gamma(s \otimes e)\gamma(e\otimes y) = (s \otimes e)(e\otimes y) = s \otimes y$. This
proves our claim, that is, the equality in Equation \ref{111} is satisfied. As a final step we want to show that
$$
\cl S \otimes_{min} C^*(\bb F_\infty) = \cl S \otimes_{max} C^*(\bb F_\infty).
$$
This again follows from the fact that the identity on $C^*(\bb F_\infty)$ factors via ucp maps on $ C^*(\bb Z_2 * \bb Z_3)$.
The proof is same as that of Theorem \ref{thm F_2Z_2}, (3)$\Rightarrow$(2). Now the DCEP criteria
given in Subsection \ref{sec DCEP} implies that $\cl S$ has DCEP.
\end{proof}

\begin{remark}
Clearly Theorem \ref{thm WEPcri} can be viewed as a corollary of Theorem \ref{thm DCEPCri} as DCEP and WEP
coincide on C*-algebras and $c$ and $max$ are the same when one of the tensorant is a C*-algebra. We proved
Theorem \ref{thm WEPcri} separately as simpler arguments are used.
\end{remark}

\begin{corollary}\label{cor WEPCriB(H)}
Let $\cl A \subset B(H)$ be a unital C*-subalgebra. Then $\cl A$ has WEP if and only if we have the
complete order embedding
\begin{eqnarray}\label{aa}
\cl A \otimes_{max} \cq{5} \subset B(H) \otimes_{max} \cq{5}.
\end{eqnarray}
\end{corollary}
\begin{proof}
By the injectivity of the minimal tensor product we readily have the complete order embedding
\begin{eqnarray}\label{aaa}
 \cl A \otimes_{min} \cq{5} \subset B(H) \otimes_{min} \cq{5}.
\end{eqnarray}
Moreover, since $\cq{5}$ has the lifting property we also have $ B(H) \otimes_{min} \cq{5} =  B(H) \otimes_{max} \cq{5}$.
Now supposing $\cl A$ has WEP, by the above theorem, we can simply replace min by max in \ref{aaa} and obtain \ref{aa}.
Conversely assuming \ref{aa} holds, combining with \ref{aaa} we get   $\cl A \otimes_{min} \cq{5} = \cl A \otimes_{max} \cq{5}$.
Thus, by the above theorem, we conclude that $\cl A$ has WEP.
\end{proof}

\begin{remark}
The above corollary can be extended as follows: Let $\cl A$ be a unital 
C*-subalgebra of a C*-algebra $\cl B$. Suppose $\cl B$ has WEP.
Then $\cl A$ has WEP if and only if we have the completely order embedding
$$
\cl A \otimes_{max} \cq{5} \subset \cl B \otimes_{max} \cq{5}.
$$
In fact we already have that $ \cl A \otimes_{min} \cq{5} \subset \cl B \otimes_{min} \cq{5} $. 
Also, as $\cl B$ has WEP, by Theorem \ref{thm WEPcri}, 
$\cl B \otimes_{min} \cq{5} = \cl B \otimes_{max} \cq{5}$. Following the same argument in the
proof of the above corollary we obtain the desired result.
\end{remark}

\begin{remark}
Any injective operator system has WEP. (Also recall that every injective operator system has a
structure of a C*-algebra.) Thus, a unital C*-algebra has WEP if and only if
 $$
\cl A \otimes_{max} \cq{5} \subset I(\cl A) \otimes_{max} \cq{5}
$$
completely order isomorphically where $I(\cl A)$ is the injective envelope of $\cl A$.
\end{remark}

The following is the four dimensional operator system variant of the Kirchberg Conjecture.

\begin{theorem}\label{thm 4dimKC}
The following are equivalent:
\begin{enumerate}
 \item The Kirchberg conjecture has an affirmative answer.
 \item $\bb C^5 / J$ has DCEP.
 \item $(\bb C^5 / J) \otimes_{min} (\bb C^5 / J) = (\bb C^5 / J) \otimes_{c} (\bb C^5 / J)$.
\end{enumerate}
\end{theorem}

For its proof we will need:
\begin{proposition}
The following are equivalent:
\begin{enumerate}
 \item The Kirchberg conjecture has an affirmative answer.
 \item $C^*(\bb Z_2 * \bb Z_3)$ has WEP.
 \item $C^*(\bb Z_2 * \bb Z_3) \otimes_{min} C^*(\bb Z_2 * \bb Z_3) =
C^*(\bb Z_2 * \bb Z_3) \otimes_{max} C^*(\bb Z_2 * \bb Z_3)$.
\end{enumerate}
\end{proposition}
\begin{proof}
(3)$\Rightarrow$(2): Follows from Theorem \ref{thm F_2Z_2}. (1)$\Rightarrow$(3): So every C*-algebra that has LLP
has WEP. In particular $C^*(\bb Z_2 * \bb Z_3)$ has WEP as it has LLP. Now (3)
follows form Example \ref{ex WEPLLP}. Finally we will prove (2)$\Rightarrow$(1). Since the identity on $C^*(\bb F_\infty)$
decomposes via ucp maps through  $C^*(\bb Z_2 * \bb Z_3)$ it follows that  $C^*(\bb F_\infty)$
has WEP, in other words, the Kirchberg conjecture has an affirmative answer.
\end{proof}

\begin{proof}[proof of Theorem \ref{thm 4dimKC}] 
Recall from last section that we can identify $\cq{5}$ with an operator subsystem of $C^*(\bb Z_2 * \bb Z_3)$
that contains enough unitaries.

(3)$\Rightarrow$(1): This is a result of Proposition \ref{pro enough}. So we have that
$$
C^*(\bb Z_2 * \bb Z_3) \otimes_{min} C^*(\bb Z_2 * \bb Z_3) =
C^*(\bb Z_2 * \bb Z_3) \otimes_{max} C^*(\bb Z_2 * \bb Z_3),
$$
in other words, the Kirchberg conjecture is true by the above proposition.

(1)$\Rightarrow$(2): Recall form the preliminaries that the Kirchberg conjecture is equivalent to statement that
every finite dimensional operator system that has the lifting property has DCEP. By Theorem \ref{thm QouhasLP}, $\cq{5}$
has the lifting property. So it must have DCEP.

(2)$\Rightarrow$(3): Since $\cq{5}$ has the lifting property it is (min,er)-nuclear. This readily shows that
$\cq{5} \otimes_{min} \cq{5} = \cq{5} \otimes_{er} \cq{5}$. Assuming (2) it is also (el,c)-nuclear.
(Applied to $\cq{5}$ on the right hand side) we get $\cq{5} \otimes_{er} \cq{5} = \cq{5} \otimes_{c} \cq{5}$.
So (3) follows.
\end{proof}

Several questions are in the order:

\begin{question}
Is $\cq{5} \otimes_{min} \cq{5} = \cq{5} \otimes_{max} \cq{5}$? 
\end{question}
\begin{question}
Is $(\cq{5} \otimes_{min} \cq{5})^+ = (\cq{5} \otimes_{max} \cq{5})^+\,$?
\end{question}
\begin{question}
Is $(\cq{5} \otimes_{min} \cq{5})^+ = (\cq{5} \otimes_{c} \cq{5})^+\,$?
\end{question}

\begin{rem}
Clearly if the first question is true then we have that the Kirchberg conjecture has an affirmative answer.
We put the second and the third questions just to emphasize the difficulty of this problem.
\end{rem}

\section{Further Properties of $\bb C^5 / J$ and Examples}

Letting $J = span\{(1,1,-1,-1,-1)\}$, the quotient operator system $\bb C^5/J$ has two important properties: 
A unital C*-algebra $\cl A$ has WEP if and only if we have the complete order isomorphism
$$
\cl A \otimes_{min} (\bb C^5 / J) = \cl A \otimes_{max} (\bb C^5 / J).
$$
Moreover, the Kirchberg conjecture has an affirmative answer if and only if we have
$$
(\bb C^5/J) \otimes_{min} (\bb C^5/J) = (\bb C^5/J) \otimes_{c} (\bb C^5/J)
$$
(equivalently $\bb C^5/J$ has DCEP ). Keeping these observation in mind it is essential to understand further nuclearity properties of
the four dimensional operator system $\bb C^5/J$. 
While Theorem \ref{thm QouhasLP} ensures that $\bb C^5/J$ has the lifting property we have:
\begin{proposition}
$\bb C^5/J$ is not exact.
\end{proposition}

\begin{proof}
By identifying $\bb Z_2 * \bb Z_3 = \langle a,b : a^2 = b^3 = e \rangle$ with its canonical representation in $C^*(\bb Z_2 * \bb Z_3)$
we have that $\cl S = span\{e,a,b,b^*\}$ is an operator subsystem of $C^*(\bb Z_2 * \bb Z_3)$. In \cite[Rem. 11.6]{Kav2} it was shown that
$\cl S$ is not exact. Since $\bb C^5/J$ and $\cl S$ are unitally and completely order isomorphic (Remark \ref{rem C5/J}), $\bb C^5/J$ is not exact.
\end{proof}

We start with the following positivity criteria. An element $x$ of an operator system $\cl S$ will be written
$x> 0 $ if $x\geq \epsilon e $ for some $\epsilon > 0$. Also, a positive element $y$ in an operator
system quotient $\cl S / J$ may not have a positive representation as it may obtain through the Archimedeanization
process. However, if $y > 0$ in $\cl S / J$ then $y - \epsilon e$ has a positive representation in $\cl S$
for some small $\epsilon > 0$.  

\begin{proposition}\label{prop pos-criteria}
Let $\cl S$ be an operator system. Take the basis $\{\dot{e}_1,\dot{e}_2,\dot{e}_3,\dot{e}_4\}$ for $\bb C^5/J$. Let
$$ 
u = s_1 \tens \dot{e}_1 + s_2 \tens \dot{e}_2 + s_3 \tens \dot{e}_3+ s_4 \tens \dot{e}_4 \in \cl S \otimes \bb C^5/J.
$$ Then:
\begin{enumerate}
 \item $u > 0$ in $ \cl S \otimes_{min} \bb C^5/J$ $\Longleftrightarrow$ $\cl S$ embeds in 
a larger operator system $\tilde{\cl S}$ ($\cl S \subset \tilde{\cl S}$) such that there is an $s \in\tilde{\cl S}^+$ with $s_1,s_2>s$ and  $s_3,s_4 \geq -s.$
 \item  $u > 0$ in $ \cl S \otimes_{c} \bb C^5/J$ $\Longleftrightarrow$ there is an $s \in C_u^*(\cl S)^+$ with $s_1,s_2>s$ and  $s_3,s_4 \geq -s.$ 
 \item  $u > 0$ in $ \cl S \otimes_{max} \bb C^5/J$ $\Longleftrightarrow$ there is an $s \in \cl S^+$ with $s_1,s_2>s$ and  $s_3,s_4 \geq -s.$
Moreover, if $\cl S$ is a finite dimensional then the strict inequalities $>$ can be taken $\geq$.
\end{enumerate}
\end{proposition}

\begin{proof}
(3): 
This is really
based on the projectivity of the maximal tensor product:
$$
(\cl S  \otimes_{max} \bb C^5)/(\cl S \otimes J) = \cl S  \otimes_{max} \cq{5}.
$$
Also, in $\cq{5}$ we have $\dot{e}_1+ \dot{e}_2 - \dot{e}_3 - \dot{e}_4 - \dot{e}_5 = 0$. Thus,
$\dot{e} = \dot{e}_1+ \dot{e}_2+\dot{e}_3+\dot{e}_4+\dot{e}_5 =  2\dot{e}_1+ 2\dot{e}_2$. Now
$u > 0   $ if and only if $ u  - \epsilon( e \tens \dot{e})$
has a positive representative in $\cl S  \otimes_{max} \bb C^5$, say $x_1\tens e_1 + x_2\tens e_2 + x_3\tens e_3 + x_4\tens e_4 + x_5\tens e_5$,
for some small $\epsilon>0$. 
Note that each of $x_i$ has to be a positive element in $\cl S$.
This means that
$$
s_1 \tens \dot{e}_1 + s_2 \tens  \dot{e}_2 + s_3 \tens  \dot{e}_3 + s_4 \tens  \dot{e}_4 - \epsilon( e \tens \dot{e}) =
x_1\tens \dot{e}_1 + x_2\tens \dot{e}_2 + x_3\tens \dot{e}_3 + x_4\tens \dot{e}_4 + x_5\tens \dot{e}_5.
$$
Now using the fact that $\dot{e}_5 = \dot{e}_1+ \dot{e}_2 - \dot{e}_3 - \dot{e}_4$ and $\dot{e}  = 2\dot{e}_1+ 2\dot{e}_2$ we obtain following
equalities:
\begin{eqnarray*}
s_1 & = & x_1 + x_5 + 2\epsilon e \\
s_2 & = & x_2 + x_5 + 2\epsilon e \\
s_3 & = & x_3 - x_5  \\
s_4 & = & x_4 - x_5.
\end{eqnarray*}
Finally putting $s = x_5$ we clearly get $s_1,s_2 > s$ and $s_3,s_4 \geq -s$ with $s \in S^+$. Note that given such elements
we can reconstruct positive elements $x_1,...,x_5$ such that the reverse direction follows. This proves the first part of (3). 
The additional part follows from the fact that $  \cl S \otimes J \subset \cl S \otimes_{max} \bb C^5$ is a completely proximinal kernel (Theorem \ref{thm max-quotient}). So every
positive element in the quotient has a positive representative in $\cl S \otimes_{max} \bb C^5$.

$ $

(1): It is enough to take $\tilde{\cl S} = B(H)$. Recall that $\bb C^{5}/J$ has the lifting
property thus we have that
$$
B(H) \otimes_{min} (\bb C^{5}/J) = B(H) \otimes_{max} (\bb C^{5}/J).
$$
Consequently, by the injectivity if the minimal tensor product, we have the embedding 
$$
\cl S \otimes_{min} (\bb C^{5}/J) \subset B(H) \otimes_{min} (\bb C^{5}/J) =  B(H) \otimes_{max} (\bb C^{5}/J).
$$
So $u > 0$ in $\cl S \otimes_{min} (\bb C^{5}/J)$ if and only if $u>0$ in $B(H) \otimes_{max} (\bb C^{5}/J)$.
Now by part (3), this is equivalent to existence of an element $s \in B(H)^+$ with
 $s_1,s_2>s$ and  $s_3,s_4 \geq -s.$ This proves (1). 
$ $

(2): From the preliminary section we have that $\cl S \otimes_{c} (\bb C^{5}/J) \subset C_u^*(\cl S) \otimes_{max} (\bb C^{5}/J)$.
Thus, $u > 0 $ in $\cl S \otimes_{c} (\bb C^{5}/J)$ if and only if $u > 0$ in $C^*(\cl S) \otimes_{max} (\bb C^{5}/J)$. Again by using
(3), the latter statement is equivalent to existence of a positive element $s$ being in $C_u^*(\cl S) $ such that
$s_1,s_2>s$ and  $s_3,s_4 \geq -s.$
\end{proof}

The sharp contrast between (1) and (3)
allows us to construct low dimensional operator systems such that the minimal and the maximal
tensor product with $\bb C^5/J$ don't coincide:

\begin{example}
$\bb C^5$ has a three dimensional operator subsystem, say $\cl S$, such that 
$$
\cl S  \otimes_{min} (\bb C^5 /J) \neq \cl S \otimes_{max} (\bb C^5 /J) .
$$
In fact for small $\epsilon>0$ consider
$$
\cl S = span\{\underbrace{(1,1,1,1,1)}_{e},\underbrace{(0,\epsilon, 1/2,1-\epsilon,1)}_{a},\underbrace{(1,0,-1/2,0,1)}_{b}\}\subset \bb C^5.
$$
We claim that $e \tens \dot{e}_1 + a \tens  \dot{e}_2 + b \tens  \dot{e}_3 + b \tens  \dot{e}_4 $ is positive in $\cl S  \otimes_{min} \cq{5}$
but not positive in $\cl S  \otimes_{max} \cq{5}$. First note that $c = (0,\epsilon,1/2,\epsilon,0) \in \bb C^5$ is positive such that
$$
e,a \geq c \mbox{ and } b \geq - c.
$$
So by the above positivity criteria ((3), additional part) we get $e \tens \dot{e}_1 + a \tens  \dot{e}_2 + b \tens  \dot{e}_3 + b \tens  \dot{e}_4 $
is positive in $\bb C^5 \otimes_{max} \cq{5} $. Since $\cl S  \otimes_{min} \cq{5}  \subset  \bb C^5 \otimes_{min} \cq{5} = \bb C^5 \otimes_{max} \cq{5}$,
it is a positive element in $\cl S  \otimes_{min} \cq{5}$. On the other hand there is no element $h \in \cl S^+$ such that $e,a \geq h$ and $b\geq -h$.
In fact we necessarily have that $h = \alpha e + \beta a + \theta b$. We leave the verification of this fact to the reader.
\end{example}

\begin{example} Let $\cl S$ be the operator system in the above example. Let $y = (1,1,-1,-1,-1)$. 
Consider the four dimensional operator subsystem $\cl T = span\{e,a,b,y\} \subset \bb C^5$. We still have
$$
\cl T  \otimes_{min} \cq{5} \neq \cl T \otimes_{max} \cq{5}.
$$
In a similar fashion we can show that  $e \tens \dot{e}_1 + a \tens  \dot{e}_2 + b \tens  \dot{e}_3 + b \tens  \dot{e}_4 $ is positive in $\cl T  \otimes_{min} \cq{5}$
but not positive in $\cl T  \otimes_{max} \cq{5}$. It is elementary to see that $\cl T\!/\!J \subset \cq{5}$. (In fact,
in general, whenever $J\subset \cl S \subset \tilde{\cl S}$, where $J$ is a kernel in $ \tilde{\cl S}$, then the induced map $\cl S / J \rightarrow  \tilde{\cl S}/J$
is a complete order embedding.) Thus we get
$$
\cl T\!/\!J \otimes_{min} \cq{5} \subset \cq{5} \otimes_{min} \cq{5}.
$$
So we direct the following question:
$$
\mbox{ Is }
\cl T\!/\!J \otimes_{min} \cq{5} =  T\!/\!J \otimes_{max} \cq{5}?
$$
\end{example}

\begin{example}
There are self-adjoint elements $s_1,s_2,s_3,s_4$ in the Calkin algebra $\bb B / \bb K$ such that
for every representation $\bb B / \bb K \subset B(K)$, where $K$ is a Hilbert space, there is a positive
element $s \in  B(K)$ with $s_1,s_2 > s$ and $s_3, s_4 \geq -s$ but there is no positive element in $\bb B / \bb K$
with these properties. This is based on the fact that $\bb B / \bb K$ does not have WEP. (The reader may refer to
\cite{Kav2} for a proof of this well-known fact.)
Thus we have
$$
\bb B / \bb K \otimes_{min} (\bb C^{5}/J) \neq \bb B / \bb K \otimes_{max} (\bb C^{5}/J).
$$
By using the C*-algebraic identification $M_n(\bb B / \bb K) \cong \bb B / \bb K$,
 this inequality fails through an element at the ground level. Now the positivity criteria (1) and (3)
imply that such elements exists in $\bb B / \bb K$. (Also note that if
$s_1,s_2 > s$ and $s_3, s_4 \geq -s$ for some element $s \in B(K)^+$ then for every
representation  $\bb B / \bb K \subset B(\tilde{K})$ there is a positive element
in $\tilde{s}$ in $B(\tilde{K})$ with these properties. This
follows from Arveson's extension theorem.)
\end{example}

We turn back to the positivity characterization given in Proposition \ref{prop pos-criteria}. Another variant can be given as follows
which will have a prominent role when we study separation properties in the next section:

\begin{proposition}
Let $\cl S$ be an operator system and $s_1,...,s_5$ be self-adjoint elements of $\cl S$. Then
there is a self-adjoint element $s \in \cl S$ such that $s_3,s_4,s_5 < s < s_1,s_2$ if and only if
$$
u = s_1 \tens \dot{e}_1 + s_2 \tens \dot{e}_2  -s_3 \tens \dot{e}_3  -s_4 \tens \dot{e}_4  -s_5 \tens \dot{e}_5 > 0
$$
in $\cl S \otimes_{max} (\bb C^5/J)$.
\end{proposition}

\begin{proof}
In $\bb C^5$ we have $\dot{e}_5 = \dot{e}_1 + \dot{e}_2- \dot{e}_3 -\dot{e}_4$. Therefore we can rewrite $u$ as
$$
u = (s_1 -  s_5) \tens \dot{e}_1 + (s_2 -  s_5) \tens \dot{e}_2 + ( - s_3 +  s_5) \tens \dot{e}_3 + (-s_4 +  s_5) \tens \dot{e}_4.
$$
Recall from Proposition \ref{prop pos-criteria} that $ u > 0$ if and only if there is an element $t \in \cl S^+$ such that
$$
s_1 -  s_5 > t,\;\; s_2 -  s_5> t, \;\;   - s_3 +  s_5 \geq -t, \;\; -s_4 +  s_5 \geq -t.
$$
Setting $s = s_5 + t$, we obtain a self-adjoint element such that
$$
s_1 > s,\;\; s_2 > s,\;\; -s_3 \geq -s ,\;\; -s_4 \geq -s,\;\; -s_5 \geq - s,
$$
equivalently $s_3,s_4,s_5 \leq s < s_1,s_2$. Consequently we obtain that $u> 0$ if and only if the latter condition
holds. It is easy to see that the latter condition is equivalent to existence of a self-adjoint where we can take the
inequalities strict (by a small $\epsilon$-perturbation of $s$ if necessary). This finishes the proof.
\end{proof}

So far we have worked with the four dimensional operator system $C^5 / J$. Similar results and positivity
criteria can be extended to a more general setting. For positive integers $k$ and $m$ we define 
$$
J_{k,m} = span\{(\underbrace{1,...,1}_{k-terms},\underbrace{-1,...,-1}_{m-terms})\} \subset \bb C^{k+m}.
$$
Before going into details we simply recall a couple of properties. $J_{k,m}$ is a one-dimensional
 null-subspace of $ \bb C^{k+m}$ so it is a completely
proximinal kernel. Moreover the quotient  $ \bb C^{k+m} / J_{k,m}$ has the lifting property as the lifting
property is stable under quotients by null-subspaces. Also by Theorem \ref{thm equiv} the following
operator systems
\begin{enumerate}
 \item $ \bb C^{k+m} / J_{k,m}$;
 \item $ \bb C^{k} \oplus_{1}  \bb C^{m}$;
 \item  $span\{e,a,a^2,...,a^{k-1}, b,b^2,...,b^{m-1}\} \subset C^*(\bb Z_{k} * \bb Z_m)$
\end{enumerate}
are unitally and completely order isomorphic.

\begin{lemma}
Suppose $k \leq k_1$ and $m \leq m_1$. Then $ \bb C^{k+m} / J_{k,m}$ can be identified
with an operator subsystem of $ \bb C^{k_1+m_1} / J_{k_1,m_1}$, moreover, this inclusion
has a ucp inverse.
\end{lemma}

\begin{proof}
Let $i$ is the embedding of $ \bb C^{k+m}$ into $\bb C^{k_1+m_1} $ given by
$$
(a_1,...,a_k,b_1,...,b_m) \mapsto (\underbrace{a_1,...,a_1}_{k_1-k \; terms},a_1,a_2,...,a_k,b_1,...,b_m,\underbrace{b_m,...,b_m}_{m_1-m \; terms})
$$
then the composition $\bb C^{k+m} \rightarrow \bb C^{k_1+m_1} \rightarrow  \bb C^{k_1+m_1} / J_{k_1,m_1}$ has the kernel $J_{m,k}$. So the
induced map $\bar{i} : \bb C^{k+m} / J_{k,m} \rightarrow \bb C^{k_1+m_1} / J_{k_1,m_1}$ is ucp. Likewise consider the
projection $q$  from $\bb C^{k_1+m_1} $ onto $ \bb C^{k+m}$ given by
$$
(a_1,...,a_{k_1},b_1,...,b_{m_1}) \mapsto (a_{k_1 - k+1},...,a_{k_1}, b_1,...,b_m).
$$
Since the composition of the ucp maps $\bb C^{k_1+m_1} \rightarrow \bb C^{k+m} \rightarrow \bb C^{k+m}/J_{k,m}$  contains  $ J_{k_1,m_1}$ in its kernel
it follows that the induced map  $\bar{q} : \bb C^{k_1+m_1} / J_{k_1,m_1} \rightarrow \bb C^{k+m} / J_{k,m}$ is ucp. It is elementary to verify
that the composition
$$
 \bb C^{k+m} / J_{k,m} \xrightarrow{\bar{i}} \bb C^{k_1+m_1} / J_{k_1,m_1}  \xrightarrow{\bar{q}} \bb C^{k+m} / J_{k,m}
$$
is the identity on $\bb C^{k+m} / J_{k,m}$. This shows that $\bar{i}$ is an embedding with ucp inverse $\bar{q}$.
\end{proof}

\begin{corollary}
If $2 \leq k$ and $3 \leq m$, $\bb C^{k+m} / J_{k,m}$ is not exact. So $C^*(\bb Z_k * \bb Z_m)$ is not exact. 
\end{corollary}

\begin{proof}
Otherwise $\bb C^{5} / J_{2,3}$ is exact as it embeds in $\bb C^{k+m} / J_{k,m}$ and exactness is stable when passing to operator
subsystems. The second part follows from Proposition \ref{pro enough}.
\end{proof}

Despite this $\bb C^4/J_{2,2}$ is C*-nuclear which we will prove soon. In the following we give a positivity criteria in $-\otimes_{max} \bb C^{k+m} / J_{k,m}$.

\begin{proposition}
Let $\cl S$ be an operator system. We take the basis $\dot{e}_1,....,\dot{e}_{k+m-1}$ for $\bb C^{k+m} / J_{k,m}$. Then
an element
$$
u = s_1 \tens \dot{e}_1 + \cdots + s_{k+m-1} \tens \dot{e}_{k+m-1} > 0
$$ 
in $\cl S \otimes_{max} (\bb C^{k+m} / J_{k,m})$ if and only if there is an element $s \in \cl S^+$ such that
$$
s < s_1,...,s_k \mbox{ and } -s \leq s_{k+1},..., s_{k+m-1}.
$$
\end{proposition}
\begin{proof}
The proof is similar to that of Proposition \ref{prop pos-criteria}.
We again use the projectivity of the maximal tensor product:
$$
(\cl S  \otimes_{max} \bb C^{k+m})/(\cl S \otimes J_{k,m}) = \cl S  \otimes_{max} (\bb C^{k+m} / J_{k,m}).
$$
In $\bb C^{k+m} / J_{k,m}$ we have $\dot{e}_{k+m} = \dot{e}_1+\cdots + \dot{e}_k - \dot{e}_{k+1} -\cdots - \dot{e}_{k+m-1}$. Consequently we obtain
$$
\dot{e} = \dot{e}_1+\cdots +\dot{e}_{k+m} = 2 \dot{e}_1+\cdots+ 2\dot{e}_k.
$$
Now $u > 0   $ if and only if $ u  - \epsilon( e \tens \dot{e})$
has a positive representative in $\cl S  \otimes_{max} \bb C^{k+m}$, say $x_1\tens e_1 + \cdots + x_{k+m}\tens e_{k+m}$,
for some small $\epsilon>0$. Clearly each of $x_i$ belongs to $\cl S^+$.
This means that
$$
s_1 \tens \dot{e}_1 + \cdots + s_{k+m-1} \tens \dot{e}_{k+m-1} - \epsilon( e \tens \dot{e}) =
x_1\tens  \dot{e}_1 + \cdots + x_{k+m}\tens  \dot{e}_{k+m}.
$$
Now using the fact that $\dot{e}_{k+m} = \dot{e}_1+\cdots + \dot{e}_k - \dot{e}_{k+1} -\cdots - \dot{e}_{k+m-1}$ and 
$ \dot{e} = \dot{e}_1+\cdots +\dot{e}_{k+m} = 2 \dot{e}_1+\cdots+ 2\dot{e}_k $  we obtain following
equalities:
$$
\begin{array}{ccc}
s_1 & = & x_1 + x_{k+m} + 2\epsilon e \\ 
\vdots &\vdots & \vdots \\
s_k & = & x_k + x_{k+m} + 2\epsilon e \\
\end{array}
\mbox{ and }
\begin{array}{ccc}
s_{k+1} & = & x_{k+1} - x_{k+m}  \\
\vdots &\vdots & \vdots \\
s_{k+m-1} & = & x_{k+m-1} - x_{k+m}  \\
\end{array}
$$
Finally putting $s = x_{k+m}$ we clearly get
$$
s < s_1,...,s_k \mbox{ and } -s \leq s_{k+1},..., s_{k+m-1}.
$$
Note that given such elements
we can reconstruct positive elements $x_1,...,x_{k+m}$ such that the reverse direction follows. 
\end{proof}

\begin{proposition}\label{pro +-poscrit.}
Let $\cl S$ be an operator system and $s_1,...,s_k$, $t_1,...,t_m$ be self-adjoint elements of $\cl S$. Then
there is a self-adjoint element $s$ in $\cl S$ such that
$$
t_1,...,t_m < s < s_1,...,s_k
$$
if and only if the following element
$$
u = s_1 \tens \dot{e}_1 + \cdots + s_k \tens \dot{e}_k - t_1 \tens \dot{e}_{k+1} - \cdots - t_m \tens \dot{e}_{k+m}
$$
is strictly positive in $\cl S \otimes_{max} \bb C^{k+m}/J_{k,m}$.
\end{proposition}

\begin{proof}
By using the fact that $\dot{e}_{k+m} = \dot{e}_1+\cdots + \dot{e}_k - \dot{e}_{k+1} -\cdots - \dot{e}_{k+m-1}$ we can re-write $u$ as
\begin{eqnarray*}
u &= &(s_1 - t_{m}) \tens \dot{e}_1 + \cdots +  (s_k - t_{m}) \tens \dot{e}_k \\
 &  &\;\;\; +  (- t_1 + t_{m}) \tens \dot{e}_{k+1} + \cdots + ( - t_{m-1} + t_{m}) \tens \dot{e}_{k+m-1}.
\end{eqnarray*}
By using the above proposition we have that $u > 0$ if and only if there is an element $t \in \cl S^+$ such that
\begin{eqnarray}\label{sss}
 -t \leq - t_1 + t_{m}, ...,  - t_{m-1} + t_{m} \mbox{ and } t < s_1 - t_{m},...,s_k - t_{m}.
\end{eqnarray}
By setting $s = t + t_{m}$ we obtain a self-adjoint element and last condition becomes
\begin{eqnarray}\label{ttt}
 t_1,...,t_{m} \leq s < s_1,...,s_k.
\end{eqnarray}
Note that the conditions in Equation \ref{sss} and \ref{ttt} are equivalent (one can simply reconstruct $t$ by setting $t = s - t_m$). Now
a simple perturbation argument shows that we can take the inequality in  Equation \ref{ttt} to be strict. This finishes the proof.
\end{proof}

We close this section by showing that $\bb C^4/J_{2,2}$ is C*-nuclear. By identifying $\bb Z_2 * \bb Z_2 = \langle a,b: a^2 = b^2 = e \rangle$
by its canonical image in $C^*(\bb Z_2 * \bb Z_2)$ we have that both $a$ and $b$ are self-adjoint
unitaries. We set 
$$
\cl S = span\{e,a,b\} \subset C^*(\bb Z_2 * \bb Z_2).
$$
The following lemma can easily be verified by the reader.
\begin{lemma}
Let $a$ be a self-adjoint element of a C*-algebra $\cl A$ such that $-e \leq a \leq e$. Then
$$
U_a = \left(
\begin{array}{cc}
a& \sqrt{e+a} \sqrt{e-a}\\
\sqrt{e+a} \sqrt{e-a} &  -a
\end{array}
\right)
$$
is a self-adjoint unitary in $M_2(\cl A)$.
\end{lemma}

\begin{lemma}
Let $\cl S \subset  C^*(\bb Z_2 * \bb Z_2)$ be the above operator subsystem and let $\cl A$ be a unital C*-algebra. 
Then every ucp map $\varphi: \cl S \rightarrow \cl A$ extends to a ucp map on $C^*(\bb Z_2 * \bb Z_2)$.
\end{lemma}

\begin{proof}
First notice that $a$ and $b$ are self-adjoint unitaries in  $C^*(\bb Z_2 * \bb Z_2)$. It is not hard to
see that $-e \leq a,b \leq e$. Let $\varphi(a) = A$ and $\varphi(b) = B$. Clearly $-e \leq A,B \leq e$ in $\cl A$.
Let $U_A$ and $U_B$ be the self-adjoint unitaries in $M_2(\cl A)$ as above. Consider the unitary representation
$$
\rho : \bb Z_2 * \bb Z_2 \rightarrow M_2(\cl A) \mbox{ given by } a \mapsto U_A \mbox{ and } b \mapsto U_B.
$$
Let $\pi:C^*(\bb Z_2 * \bb Z_2) \rightarrow M_2(\cl A)$ be the unital $*$-homomorphism extending $\rho$. It is easy to see
that $\psi: M_2(\cl A) \rightarrow \cl A$, $(A_{ij}) \mapsto A_{11}$ is a ucp map. Thus $\psi \circ \pi$ is a ucp map
from $C^*(\bb Z_2 * \bb Z_2)$ into $\cl A$. Note that $(\psi \circ \pi)(a) = A$ and $(\psi \circ \pi)(b) = B$. So it extends
$\varphi$. 
\end{proof}

By using the fact that $ C^*(\bb Z_2 * \bb Z_2)$ is nuclear we can deduce the following:
\begin{corollary}
The above operator subsystem $\cl S \subset  C^*(\bb Z_2 * \bb Z_2)$ is C*-nuclear.
\end{corollary}
\begin{proof}
Let $\cl A$ be a C*-algebra. We first claim that every ucp map $\varphi : \cl S \otimes_{max} \cl A \rightarrow B(H)$
extends to a ucp map on $C^*(\bb Z_2 * \bb Z_2) \otimes_{max} \cl A$. This is enough to conclude that we have the
complete order embedding $\cl S \otimes_{max} \cl A \subset C^*(\bb Z_2 * \bb Z_2) \otimes_{max} \cl A$.
Since $c$ and $max$ coincides when one of the
tensorant is a C*-algebra we can think of $max$ as $c$. This means that there is a Hilbert space $K \supset H$,
ucp maps $\phi : \cl S \rightarrow B(K)$ and $\psi : \cl A \rightarrow B(K)$ with commuting ranges such that $\varphi = V^*( \phi \cdot \psi )V$,
where $V$ is the inclusion of $H$ in $K$ and so $V^*$ is the projection onto $H$. By using the above lemma, let $\tilde{\phi}: C^*(\bb Z_2 * \bb Z_2) \rightarrow B(K)$
be the ucp extension of $\phi$ such that the image of $\tilde{\phi}$ stays in the C*-algebra generated by $\phi(\cl S)$. 
Clearly $\tilde{\phi}$ and $\psi$ still have the commuting ranges.
This means that $\tilde{\phi} \cdot \psi$ is a ucp map from $C^*(\bb Z_2 * \bb Z_2) \otimes_{max} \cl A$ into $B(K)$. The compression
of this map, i.e., $V^* (\tilde{\phi} \cdot \psi)V$ is  ucp and it extends $\varphi$. This proves  our claim. Finally by the injectivity
of the minimal tensor product we have that  $\cl S \otimes_{min} \cl A \subset C^*(\bb Z_2 * \bb Z_2) \otimes_{min} \cl A$.
Now since $C^*(\bb Z_2 * \bb Z_2) $ is nuclear we obtain $\cl S \otimes_{min} \cl A = \cl S \otimes_{max} \cl A$.
Since $\cl A$ was arbitrary,  $\cl S$ is C*-nuclear.
\end{proof}

\begin{corollary}
$\bb C^4 / J_{2,2}$, i.e., $\bb C^4/ span\{(1,1,-1,-1)\}$, is C*-nuclear.
\end{corollary}
\begin{proof}
This is a consequence of the identification in Theorem \ref{thm equiv}. $\bb C^4 / J_{2,2}$ and the operator system $\cl S$ in the
above corollary are unitally completely order isomorphic.
\end{proof}

\begin{example}
Let $\cl A \subset B(H)$ be a unital C*-subalgebra. Let $a,b,c \in \cl A$ such that
there is an element $x \in B(H)^+$ with $a,b > x$ and $c \geq -x$. Then there
is an element in $\cl A^+$ with these properties, that is, there is $y \in \cl A^+$ with
$a,b>y$ and $c \geq -y$. To see this first note that
$$
\cl A \otimes_{max} (\bb C^4 / J_{2,2}) \subset B(H) \otimes_{max} (\bb C^4 / J_{2,2}) 
$$
which follows from the injectivity of the minimal tensor product and C*-nuclearity of  $ (\bb C^4 / J_{2,2}) $. Now if take the
basis $\{\dot{e}_1,\dot{e}_2,\dot{e}_3\}$ for $ (\bb C^4 / J_{2,2}) $ then the element $ u = a \tens \dot{e}_1 + b \tens \dot{e}_2 +c \tens \dot{e}_3$
is strictly positive in $B(H) \otimes_{max} (\bb C^4 / J_{2,2}) $ if and only if there is an element $x \in B(H)^+$ with $a,b > x$ and $c \geq -x$ (Proposition \ref{prop pos-criteria}).
So the condition we assumed is equivalent to $u > 0$. The above unital complete order embedding ensures that $u > 0$ in  $\cl A \otimes_{max} (\bb C^4 / J_{2,2})$ too.
So such an element exists in $\cl A$.
\end{example}


\section{``Relative'' Tight Riesz Interpolation Property}

In this section we present an equivalent formulation of Lance's weak expectation property in terms
of a separation property in matricial order structure of a C*-algebra $\cl A$  which corresponds
to a non-commutative Riesz interpolation property.  Though our construction makes use of quotient and tensor
theory of operator systems we restrict our application
on C*-algebras as the weak expectation property is well known characteristic in this context.

\begin{definition}
Let $\cl A $ be a unital C*-subalgebra of a C*-algebra $\cl B$. 
We say that $\cl A$ has the \textit{$(k,m)$ tight Riesz interpolation property} 
in $\cl B$, TR$(k,m)$-property in short,
if for any $x_1,...,x_k$ and $y_1,...,y_m$ in $\cl A_{sa}$ if there exists 
an element  $b \in B_{sa}$ with
$$
x_1,...,x_k < b < y_1,...,y_m
$$
 then there is an element $a \in A_{sa}$
such that 
$$
x_1,...,x_k < a < y_1,...,y_m.
$$
We say that $\cl A$ has 
the \textit{complete $(k,m)$ tight Riesz interpolation property} in $\cl B$
if $M_n(\cl A) $ has  TR$(k,m)$-property in $M_n(\cl B)$ for every $n$ and we 
abbreviate the latter condition as complete TR$(k,m)$-property.
\end{definition}

\begin{rem}
In lattice group theory a group $G$ is said to have TR$(k,m)$-property 
if for any  $x_1,...,x_k$ and $y_1,...,y_m$ in $G$ with $x_i< y_j$ for all
$i$ and $j$ there is an element $g\in G$ such that $x_i<g< y_j$ for all
$i$ and $j$. We remark that in this sense even the additive abelian group of selfadjoints
in $M_2$ fails to have TR$(2,2)$-property. The following example is pointed out by Vern Paulsen. Consider 
$$
a = \left(
\begin{array}{cc}
 1 & 0 \\
0 & 0
\end{array}
\right),
b = \left(
\begin{array}{cc}
 1 & 1 \\
1 & 1
\end{array}
\right),
c = \left(
\begin{array}{cc}
 1 .1 & 0.5 \\
0.5 & 3.6 
\end{array}
\right),
\mbox{ and }
d = \left(
\begin{array}{cc}
 3.6 & 0.5 \\
0.5 & 1.1
\end{array}
\right).
$$
It follows that $a,b < c,d$ but there is no element $x$ in $M_2$ with $a,b < x<c,d$.
We leave the verification of this to the reader.
Consequently, our understanding
in the present paper is \textit{relative} TR$(k,m)$-property, that is, if  such
pairs have a separation in a larger object then whether it exists in the smaller object.
\end{rem}

We mostly interested in this phenomena when $\cl B = B(H)$ for the following reason:
Let $\cl A$ be a C*-subalgebra of $B(H)$ and  let $x_1,...,x_k$ and $y_1,...,y_m$ be
self-adjoint elements of $\cl A$. Suppose that $\cl A$ embeds into an operator
system $\cl S$ unitally and completely order isomorphically (via $i$) such that
 there exists 
an element  $s \in \cl S_{sa}$ with
$$
i(x_1),...,i(x_k) < s <i(y_1),...,i(y_m).
$$
 Then there is an element $b \in B(H)$
such that 
$$
x_1,...,x_k < b < y_1,...,y_m.
$$
In other words, whenever interpolation exists in a representation of $\cl A$ then it exists in $B(H)$. In fact,
by Arveson's extension theorem  \cite{Ar1}, the inclusion $\cl A \subset B(H)$ extends to a ucp map from $\cl S$ into $B(H)$, say $\varphi$.
It is elementary to see that $b = \varphi(s)$ has the desired property.

\begin{theorem}\label{thm TR-tens}
Let $\cl A$ be a unital C*-subalgebra of $\cl B$. Then the following are equivalent:
\begin{enumerate}
 \item $\cl A$ has the complete TR$(k,m)$-property in $\cl B$.
 \item We have a unital and complete order embedding
$$
\cl A \otimes_{max} (\bb C^{k+m}/J_{k,m}) \subset \cl B \otimes_{max} (\bb C^{k+m}/J_{k,m}).
$$
\end{enumerate} 
\end{theorem}
\begin{proof}
Before starting the proof note that if we remove the word ``complete'' in (1) and ``completely'' in (2)
and prove the result this way then the original statement automatically satisfied. This simply follows
from the associativity of the maximal tensor product which yields: $M_n(\cl S \otimes \cl T) = M_n( \cl S) \otimes_{max} \cl T$
for every operator systems $\cl S$, $\cl T$ and $n$. Secondly, for the compatibility with our previous results,
we remark that $\cl A$ has TR$(k,m)$-property in $\cl B$ if and only if it has TR$(m,k)$-property in $\cl B$.

Turning back to proof first suppose (2). Let $x_1,...,x_m$ and $y_1,...,y_k$ be elements in $\cl A_{sa}$
such that there is an element in $b$ in $\cl B_{sa}$ with $x_i<b<y_j$ for all $i=1,...,m$ and
$j=1,...,k$. Py Proposition \ref{pro +-poscrit.} we get
$$
u = y_1\tens  \dot{e}_1 +  \cdots + y_k \tens \dot{e}_k - x_1 \tens \dot{e}_{k+1} - \cdots -x_m \tens \dot{e}_{k+m} > 0
$$
in $ \cl B \otimes_{max} (\bb C^{k+m}/J_{k,m})$. Since we assumed (2), $u > 0$ in $\cl A \otimes_{max} (\bb C^{k+m}/J_{k,m})$ too.
Again by using Proposition \ref{pro +-poscrit.}, there is an element $a \in \cl A_{sa}$ with $x_i<a<y_j$ for all $i$ and $j$. Conversely
suppose that $\cl A$ has TR$(m,k)$-property in $\cl B$. We take the basis $\{\dot{e}_1,...,\dot{e}_{k+m-1}\}$ for $\bb C^{k+m}/J_{k,m}$.
To establish the order embedding between the tensor products we need to show that if 
$$
v = a_1 \tens \dot{e}_1  + \cdots + a_{k+m-1} \tens \dot{e}_{k+m-1} \geq 0
$$
in $\cl B \otimes_{max} (\bb C^{k+m}/J_{k,m})$, where $a_1,...,a_{k+m-1} \in \cl A$, then 
$v \geq 0$ in $\cl A \otimes_{max} (\bb C^{k+m}/J_{k,m})$. A moment of thought shows that we can replace $\geq$
by $>$. So suppose $v > 0$. Note that $a_1,...,a_{k+m-1}$ must be self-adjoint elements of $\cl A$. By simply
adding the term $ 0  \tens \dot{e}_{k+m}$ to $v$ and using Proposition \ref{pro +-poscrit.} we see that there is an element $b$ in $\cl B_{sa}$
such that $0, -  a_{k+m-1},..., -a_{k+1} < b < a_1,...,a_k$. Since we assumed TR$(m,k)$-property there is an element $a \in \cl A_{sa}$
such that $0, -  a_{k+m-1},..., -a_{k+1} < a < a_1,...,a_k$. Again by using Proposition \ref{pro +-poscrit.} $v> 0$ in $\cl A \otimes_{max} (\bb C^{k+m}/J_{k,m})$.
\end{proof}

We are ready to state our main results. 

\begin{theorem}
Let $\cl A \subset B(H)$ be a unital C*-subalgebra. Then $\cl A$ has the complete TR(2,2)-property in $B(H)$.
\end{theorem}

\begin{proof}
By Theorem \ref{thm TR-tens},  $\cl A$ has the complete TR(2,2)-property in $B(H)$ if and only if we have
$$
\cl A \otimes_{max} (\bb C^{4}/J_{2,2}) \subset B(H) \otimes_{max} (\bb C^{4}/J_{2,2}).
$$
Since $\bb C^{4}/J_{2,2}$ is C*-nuclear we can replace $max$ by $min$ so the latter condition holds.
\end{proof}

Here is the extension of this result.

\begin{theorem}
Let $\cl A \subset B(H)$ be a unital C*-algebra. Then the following are equivalent:
\begin{enumerate}
 \item $\cl A$ has WEP.
 \item $\cl A$ has the complete TR$(2,3)$-property in $B(H)$.
 \item $\cl A$ has the complete TR$(k,m)$-property in $B(H)$ for some $k \geq 2$ and $m \geq 3$.
 \item $\cl A$ has the complete TR$(k,m)$-property in $B(H)$ for all integers $k,m \geq 1$.
\end{enumerate}
\end{theorem}

\begin{proof}
(4)$\Rightarrow$(3) is clear. To prove (3)$\Rightarrow$(2) one can simply use the definition of the tight Riesz interpolation.
We will show (2)$\Rightarrow$(1). By Theorem \ref{thm TR-tens}, (2) is equivalent to
$$
\cl A \otimes_{max} (\bb C^{5}/J_{2,3}) \subset B(H) \otimes_{max} (\bb C^{5}/J_{2,3}).
$$
Recall from Corollary \ref{cor WEPCriB(H)} that this complete order embedding is equivalent to $\cl A$ having WEP. Finally suppose (1).
To obtain (4), by considering Theorem \ref{thm TR-tens}, we need to show that
\begin{eqnarray}\label{asdf}
\cl A \otimes_{max} (\bb C^{k+m}/J_{k,m}) \subset B(H) \otimes_{max} (\bb C^{k+m}/J_{k,m})
\end{eqnarray}
for every $k,m \geq 1$. Note that since $\bb C^{k+m}/J_{k,m}$ has the lifting property we have that
$$
 B(H) \otimes_{min} (\bb C^{k+m}/J_{k,m})  =  B(H) \otimes_{max} (\bb C^{k+m}/J_{k,m}).
$$
Also using the fact that WEP is equivalent to (el,max)-nuclearity and lifting property is equivalent to (min,er)-nuclearity
(and considering $\bb C^{k+m}/J_{k,m}$ is on the right-hand side) we have
$$
\cl A \otimes_{min} (\bb C^{k+m}/J_{k,m}) = \cl A \otimes_{el} (\bb C^{k+m}/J_{k,m}) = \cl A \otimes_{max} (\bb C^{k+m}/J_{k,m}).
$$
So both of the maximal tensor products in Equation \ref{asdf} can be replaced by $min$. Since the minimal tensor product is injective the result follows.
\end{proof}

Starting with a unital C*-algebra with WEP we can characterize its C*-subalgebras that have WEP.

\begin{corollary}
Let $\cl A$ be a unital C*-algebra of $B$. Suppose that $\cl B$ has WEP. Then $\cl A$ has WEP if and only if
$\cl A$ has the complete TR$(2,3)$-property in $\cl B$. 
\end{corollary}

\begin{proof}
By Theorem \ref{thm WEPcri}, $\cl A$ has WEP if and only if $$\cl A \otimes_{min}(\bb C^5/J) = \cl A \otimes_{max}(\bb C^5/J).$$
We readily have that $$\cl B \otimes_{min}(\bb C^5/J) = \cl B \otimes_{max}(\bb C^5/J).$$ So $\cl A$ has WEP if and only if we have the
complete order embedding
$$
 \cl A \otimes_{max}(\bb C^5/J) \subset  \cl B \otimes_{max}(\bb C^5/J).
$$
By Theorem \ref{thm TR-tens}, this is equivalent to $\cl A$ having complete TR(2,3)-property in $\cl B$.
\end{proof}

\begin{corollary}
A unital C*-algebra has WEP if and only if it has the complete TR(2,3)-property in its injective envelope $I(\cl A)$.
\end{corollary}

\begin{proof}
Follows from the fact that every injective operator system has WEP. (This is elementary to see by using the definition of WEP.)
\end{proof}

\begin{corollary}
Let $k,m$ be positive integers. Then every unital C*-algebra $\cl A$ has the complete TR(k,m)-property in $\cl A^{**}$.
\end{corollary}
\begin{proof}
By \cite[Lem. 6.5.]{kptt2} we have the unital complete order embedding
$$
 \cl A \otimes_{max}(\bb C^{k+m}/J_{k,m}) \subset  \cl A^{**} \otimes_{max}(\bb C^{k+m}/J_{k,m}).
$$
So Theorem \ref{thm TR-tens} yields the desired conclusion.
\end{proof}

\begin{corollary}
 Let $\cl A$ be a  unital C*-algebra with WEP. Then for every unital C*-algebraic inclusion $\cl A \subset \cl B$,
$\cl A$ has the complete TR(k,m)-property in $\cl B$, for any  $k,m \geq 1$.
\end{corollary}

\begin{proof}
Given positive integers $k$ and $m$, we need to show that
$$
 \cl A \otimes_{max}(\bb C^{k+m}/J_{k,m}) \subset  \cl B\otimes_{max}(\bb C^{k+m}/J_{k,m}).
$$
Let $\tau$ be the operator system structure on the algebraic tensor product $\cl A \otimes(\bb C^{k+m}/J_{k,m}) $
induced by $ \cl B\otimes_{max}(\bb C^{k+m}/J_{k,m})$. Clearly $min \leq \tau \leq max$. Since the minimal
and the maximal tensor product of $\cl A$ and $\bb C^{k+m}/J_{k,m}$ coincide (Theorem \ref{thm WEPcri}), we get
$min = \tau = max$. Simply replacing $\tau$ with $max$ we obtain that the above embedding holds.
\end{proof}

\subsection{Ordered spaces} In ordered function space theory (let us restrict ourself with Kadison's
function spaces in real case \cite{kad} or AOU spaces in complex case -in the sense of Paulsen and Tomforde \cite{pt}) a space $V$
is said to have TR$(k,m)$-property if for every $v_1,...,v_k$ and $w_1,...,w_m$ with $v_i < w_j$ for all
$i = 1,...,k$ and $j = 1,...,m$ there is an element $v$ such that
$$
v_1,...,v_k < v < w_1,...,w_m.
$$
We remark that a function space can be thought concretely as a unital real subspace of $C_{\bb R}(X)$.
Likewise, an AOU space can be considered concretely as a unital subspaces of $C(X)$ which is closed
under the involution. Note that the conditions
on $v_1,...,v_k$ and $w_1,...,w_m$, i.e., $v_i < w_j$ for all
$i = 1,...,k$ and $j = 1,...,m$, is equivalent to a separation in a larger object, namely $C_{\bb R}(X)$ or $C(X)$. 
Since the least upper bound of $v_1,...,v_k$, say $v$,  trivially exists in $C_{\bb R}(X)$ (or $C(X)$)
(which, indeed, makes $C_{\bb R}(X)$ and $C(X)$ a Riesz space or a vector lattice),
a small perturbation $\tilde v$ of $v$ satisfies $v_1,...,v_k <\tilde v < w_1,...,w_m.$
Therefore our definition of ``relative'' TR($k,m$)-property carries out the same character 
to C*-algebra theory. We refer the reader to \cite{NP} for the tensorial aspects of  Riesz interpolation properties in ordered function space theory.


\begin{thebibliography}{9}


\bibitem{Ar1} W. B. Arveson, {\it Subalgebras of C*-algebras I}, Acta Math. {\bf 123} (1969) 141-224.  


\bibitem{BEF} B.V.R. Bhat, G.A. Elliott, P.A. Fillmore, Editors , \textit{Lectures on Operator Theory}, Fields Institute Monographs \textbf{13}, American Mathematical Society, Providence (1999)




\bibitem{blecher} D. P. Blecher and  B. L. Duncan, {\it Nuclearity-related properties for nonselfadjoint algebras},
J. Operator Theory {\bf 65}  (2011) 47-70.


\bibitem{bo} F. P. Boca, \textit{A note on full free product C*-algebras, lifting and quasidiagonality}, Operator
theory, operator algebras and related topics (Timi¸soara, 1996), Theta Found., Bucharest,
1997, pp. 51-63.



\bibitem{CE2} M. D. Choi and E. G. Effros, {\it Injectivity and operator spaces}, J. Funct. Anal. {\bf 24} (1977)
156-209.


\bibitem{CE0} M. D. Choi and E. G. Effros, {\it Nuclear C*-algebras and the approximation property}, Amer. J.
Math. {\bf 100} (1978) 61-79.



\bibitem{pf} D. Farenick and  V. I. Paulsen, \textit{Operator system quotients of matrix algebras and their tensor products}, arXiv:1101.0790.



\bibitem{TF} T. Fritz, {\it Operator system structures on the unital direct sum
of C*-algebras}, 2010, arXiv:1011.1247.



\bibitem{Han} K. H. Han, {\it On maximal tensor products and quotient maps of operator systems}, arXiv:1010.0380v2


\bibitem{HP} K. H. Han and V. Paulsen,  {\it An approximation theorem for nuclear operator systems}, J. Funct.
Anal. {\bf 261} (2011), 999-1009.




\bibitem{PDLH} P. de La Harpe, {\it Topics in geometric group theory}, The University of Chicago Press, (2000)


\bibitem{kad} R.V. Kadison, \textit{A representation theory for commutative topological algebra}, Mem. Amer. Math.
Soc., 1951, (1951). no. 7.



\bibitem{Kav} A.S. Kavruk, \textit{Tensor products of operator systems and applications}, Ph.D. thesis, Univ. of Houston,
2011.

\bibitem{Kav2} A.S. Kavruk, \textit{Nuclearity Related Properties in Operator Systems}, arXiv:1107.2133





\bibitem{kptt2} A. S. Kavruk, V. I. Paulsen, I. G. Todorov and M. Tomforde,
\textit{Quotients, exactness and nuclearity in the operator system category},
preprint.


\bibitem{kptt} A. S. Kavruk, V. I. Paulsen, I. G. Todorov and M. Tomforde, {\it Tensor
products of operator systems}, J. Funct. Anal., {\bf 261} (2011) 267-299.


\bibitem{KL} D. Kerr and H. Li {\it On Gromov-Hausdorff convergence for operator metric spaces}, 
J. Operator Theory, {\bf 62} (2009) 83-109.


\bibitem{Ki2} E. Kirchberg, \textit{On non-semisplit extensions, tensor products
and exactness of group C*-algebras}, Invent. Math. {\bf 112} (1993) 449 - 489.



\bibitem{Ki1} E. Kirchberg, {\it On subalgebras of the CAR-algebra}, J. Functional Anal. {\bf 129} (1995), 35-63.




\bibitem{La2} C. Lance, \textit{On nuclear C*-algebras}, J. Funct. Anal. {\bf 12} (1973) 157-176.

\bibitem{La} C. Lance, {\it Tensor products and nuclear C*-algebras}, Proceedings of Symposia in Pure
Mathematics, Vol. {\bf 38} (1982) 379-399.




\bibitem{NP} I. Namioka and R.R. Phelps, \textit{Tensor products of compact convex sets}, Pasific J. of Math. Vol. \textbf{31}, No. 2, (1969)


\bibitem{Pa}
V. I. Paulsen, {\it Completely bounded maps and operator
algebras}, Cambridge Studies in Advanced Mathematics {\bf 78}, {\rm Cambridge
University Press, 2002.}




\bibitem{pt}
V. I. Paulsen and M. Tomforde, {\it Vector spaces with an
order unit}, {\rm Indiana Univ. Math. J.,} to appear.



\bibitem{pi} G. Pisier, {\it Introduction to operator space theory},
  London Mathematical Society Lecture Note Series {\bf 294}, {\rm
    Cambridge University Press, 2003.}



\bibitem{Was2} S. Wassermann, {\it The slice map problem for C*-algebras}, Proc. London Math. Soc. {\bf 32} (1976) 537-559.


\end{thebibliography}
\end{document}